\documentclass[10pt,a4paper]{article}
\usepackage[margin=3.24cm]{geometry}
\usepackage{amsfonts,amsmath,bm,dsfont,mathrsfs}
\usepackage{amssymb,amsthm}
\usepackage{enumitem,cases}
\usepackage[dvipsnames]{xcolor}
\colorlet{xlinkcolor}{red!50!black}
\usepackage{hyperref}[6.83]
\definecolor{refblue}{RGB}{26,13,171}
\hypersetup{
  colorlinks = true,
  allcolors = refblue,
  urlcolor = BrickRed,
}

\newtheorem{definition}{Definition}

\newtheorem{lemma}{Lemma}
\newtheorem{example}{Example}
\newtheorem{theorem}{Theorem}
\newtheorem{proposition}{Proposition}
\newtheorem{corollary}{Corollary}
\newtheorem{remark}{Remark}
\numberwithin{equation}{section}
\numberwithin{lemma}{section}
\numberwithin{example}{section}
\numberwithin{definition}{section}
\numberwithin{assumption}{section}
\numberwithin{theorem}{section}
\numberwithin{proposition}{section}
\numberwithin{corollary}{section}
\numberwithin{remark}{section}
\DeclareMathOperator*{\argmin}{arg\,min}
\newcommand\aff{{\sf aff}}

\newcommand{\conv}{{\sf conv}}

\newcommand{\interior}{{\sf int}}
\newcommand{\lin}{{\sf lin}}
\renewcommand{\ker}{{\sf ker}}
\newcommand{\range}{{\sf rge}}
\newcommand{\gph}{{\sf gph}}
\newcommand\Span{{\sf span}}
\newcommand{\trace}{{\sf trace}}
\renewcommand\Re{{\mathds R}}

\newcommand{\closure}{{\rm cl}}
\newcommand{\epi}{{\rm epi}}
\newcommand\rS{{\mathbb S}}

\newcommand\KKT{{\rm KKT}}
\newcommand{\ds}{\displaystyle}

\def\[{\begin{equation}}
\def\]{\end{equation}}
\def\cA{{\mathcal A}}
\def\cB{{\mathcal B}}
\def\cC{{\mathcal C}}
\def\cD{{\mathcal D}}
\def\cE{{\mathcal E}}
\def\cF{{\mathcal F}}

\def\cH{{\mathcal H}}

\def\cK{{\mathcal K}}
\def\cL{{\mathcal L}}
\def\cM{{\mathcal M}}
\def\cN{{\mathcal N}}

\def\cQ{{\mathcal Q}}

\def\cS{{\mathcal S}}
\def\cT{{\mathcal T}}
\def\cU{{\mathcal U}}
\def\cV{{\mathcal V}}
\def\cW{{\mathcal W}}
\def\cX{{\mathcal X}}
\def\cY{{\mathcal Y}}

\setlist[description]{style=multiline,leftmargin=2em}
\setlist[itemize]{style=standard,leftmargin=2em}
\setlist[enumerate]{style=standard,leftmargin=2.2em,itemsep=3pt}

\title{\Large\bf
Characterizations of the Aubin Property of the Solution Mapping for Nonlinear Semidefinite Programming\thanks{This work was funded by
the National Key R \& D Program of China (No. 2021YFA001300),
the National Natural Science Foundation of China (Nos. 12271150 and 12171271),
the Hunan Provincial Natural Science Foundation of China (No. 2023JJ10001),
the Science and Technology Innovation Program of Hunan Province (No. 2022RC1190),
the RGC Senior Research Fellow Scheme (No. SRFS2223-5S02),
and the GRF Projects (Nos. 15307822 and 15307523).}}
\author{
Liang Chen\thanks{School of Mathematics, Hunan University, Changsha, China (\url{chl@hnu.edu.cn}).}
\quad   
Ruoning Chen\thanks{
Department of Mathematical Sciences, Tsinghua University, Beijing, China; and 
Department of Applied Mathematics, The Hong Kong Polytechnic University, Hung Hom, Hong Kong (\url{crn22@mails.tsinghua.edu.cn},
\url{ruoning.chen@polyu.edu.hk}).}
\ 
\quad 
Defeng Sun\thanks{Department of Applied Mathematics, The Hong Kong Polytechnic University, Hung Hom, Hong Kong (\url{defeng.sun@polyu.edu.hk}).}
\ 
\quad 
Liping Zhang\thanks{Department of Mathematical Sciences, Tsinghua University, Beijing, China (\url{lipingzhang@tsinghua.edu.cn}).}
}
\date{January 7, 2025} 
\begin{document}
\maketitle
 
\begin{abstract}
In this paper, we study the Aubin property of the Karush-Kuhn-Tucker solution mapping for the nonlinear semidefinite programming (NLSDP) problem at a locally optimal solution. 
In the literature, it is known that the Aubin property implies the constraint nondegeneracy by Fusek [SIAM J. Optim. 23 (2013), pp. 1041-1061] and the second-order sufficient condition by Ding et al. [SIAM J. Optim. 27 (2017), pp. 67-90]. 
Based on the Mordukhovich criterion, here we further prove that the strong second-order sufficient condition is also necessary for the Aubin property to hold.
Consequently, several equivalent conditions including the strong regularity are established for NLSDP's Aubin property.
Together with the recent progress made by Chen et al. on the equivalence between the Aubin property and the strong regularity for nonlinear second-order cone programming [SIAM J. Optim., in press; arXiv:2406.13798v3 (2024)], this paper constitutes a significant step forward in characterizing the Aubin property for general non-polyhedral $C^2$-cone reducible constrained optimization problems.
 

\medskip
\noindent 
{\bf MSCcodes:}
49J53, 90C22, 90C31, 90C46 
\end{abstract}

\section{Introduction}\label{sec:intro}
Consider the constrained optimization problem
\[
\label{op}
\min_{x\in\cX} \,f(x)
\quad\mbox{s.t.}
\quad
G(x)\in\cK,
\]
where $\cX$ and $\cY$ are two finite-dimensional real Hilbert spaces each endowed with an inner product $\langle \cdot,\cdot\rangle$ and its induced norm $\|\cdot\|$,
$f:\cX\to\Re$ and $G:\cX\to\cY$ are twice continuously differentiable functions, and $\cK\subseteq\cY$ is a closed convex set.
The Lagrangian function of \eqref{op} is defined by
\[
\label{lagrangian}
\cL(x,y):=f(x)+\langle y, G(x)\rangle,
\quad
(x; y)\in\cX\times\cY. 
\]
Here and throughout this paper, the notation $(\cdot;\cdot)$ signifies the arrangement of two vectors or linear operators symbolically column-wise.
Then the first-order optimality condition of \eqref{op} is characterized by the Karush-Kuhn-Tucker (KKT) system
\[
\label{kktop}
0=\nabla_{x}\cL(x,y)
\quad\mbox{and}\quad  
y\in\cN_{\cK}(G(x)),
\]
where $\nabla_x \cL(x,y)$ denotes the adjoint of $\cL'_x(x,y)$, the partial derivative of $\cL$ with respect to $x\in\cX$,
and  $\cN_{\cK}$ denotes the normal cone of $\cK$ in convex analysis \cite{rocconv}.
For any solution $(\bar x;\bar y)\in\cX\times\cY$ of the KKT system \eqref{kktop}, 
we say that $\bar x$ is a stationary point of  \eqref{op} and $\bar y$ is a (Lagrange) multiplier at $\bar x$.
The set of all multipliers at $\bar x$ is denoted by $\cM(\bar x)$. 

The canonically perturbed version \cite[Section 5.1.3]{pert} of \eqref{op} is given by
\[
\label{pop}
\min_{x\in\cX} \,f(x)-\langle a,x\rangle
\quad\mbox{s.t.}
\quad
G(x)+b\in\cK,
\]
where $a\in\cX$ and $b\in\cY$ are the perturbation parameters.
With $\cL$ being the Lagrangian function defined by \eqref{lagrangian},
the KKT system of \eqref{pop}, as a perturbed KKT system of \eqref{op}, is given by
\[
\label{kktpop}
a=\nabla_{x}\cL(x,y)\quad\mbox{and}\quad
y\in\cN_{\cK}(G(x)+b).
\]
Then, one can associate \eqref{kktpop} with the solution mapping 
\[
\label{skkt}
{\rS}_{\KKT}(a,b)
:=
\{(x;y)\mid
a=\nabla_{x}\cL(x,y),\
y\in\cN_{\cK}(G(x)+b)
\}.
\]
As a core research topic in perturbation analysis of optimization problems, how the solution set ${\rS}_{\KKT}(a,b)$ of \eqref{kktpop} varies along with $(a;b)$ around the origin has been studied for a long time \cite{pert,Dontbook,facbook,ho2008,klabook,pangsunsun,varbook}. 
In a landmark paper of Robinson \cite{Robinson1980}, the definition of strong regularity was introduced to extend the inverse and implicit function theorems to generalized equations.
Note that the KKT system \eqref{kktop} can be equivalently reformulated to the generalized equation 
\begin{equation}\label{def:phi}
0\in
\Phi(x,y):=
\begin{pmatrix}
\nabla_x\cL(x,y)
\\ -G(x)
\end{pmatrix}+
\begin{pmatrix}
\{0\}
\\ 
\cN_{\cK}^{-1}(y) 
\end{pmatrix}. 
\end{equation}
Then, we say that $(\bar x;\bar y)\in\cX\times\cY$ is a strongly regular solution of \eqref{def:phi} (or the KKT system \eqref{kktop}) if the inverse of the set-valued mapping 
\begin{equation*}
\overline\Phi(x,y):=\begin{pmatrix}
\nabla_{xx}^2\cL(\bar x,\bar y)(x-\bar x)+\nabla G(\bar x)(y-\bar y)\\-G(\bar x)- G'(\bar x)(x-\bar x)
\end{pmatrix}+
\begin{pmatrix}
\{0\} \\ \cN_{\cK}^{-1}(y)
\end{pmatrix}
\end{equation*}
has a Lipschitz continuous single-valued localization around $(0; 0)\in\cX\times\cY$ for $(\bar x;\bar y)$, where 
$\nabla_{xx}^2\cL$ is the partial Hessian of $\cL$ with respect to $x$. 
According to \cite{Dontchev95},
such a strong regularity condition is equivalent to the condition that the solution mapping ${\rS}_{\rm KKT}$ 
has a Lipschitz continuous single-valued localization around
$((0;0);(\bar x;\bar y))$, or equivalently, the mapping
$\Phi$ in \eqref{def:phi} is strongly metrically regular at $(\bar x;\bar y)$ for $(0;0)$. 
Another significant yet less restrictive concept in studying the variation of  
${\rS}_{\rm KKT}(a,b)$ with respect to the perturbation is the Aubin property (cf. \cite[Section 9F]{varbook} or \cite[Section 3E]{Dontbook} for a systematic introduction), 
which was originally called the ``pseudo-Lipschitzian'' property by Aubin \cite{aubin}.
The Aubin property of ${\rS}_{\rm KKT}$ holds at $(0;0)\in\cX\times\cY$ for
$(\bar x;\bar y)\in\rS_{\KKT}(0,0)$ if there exist a constant $\kappa>0$ and open neighborhoods $\cU$ of $(0;0)$ and $\cV$ of $(\bar x;\bar y)$ such that
$$
{\rS}_{\rm KKT}(a',b')\cap \cV\subseteq{\rS}_{\rm KKT}(a,b)+\kappa\|(a';b')-(a;b)\|\mathds{B}_{\cX\times\cY} 
\  \,
\forall \, (a;b), (a';b')\in \cU,
$$
where $\mathds{B}_{\cX\times\cY}$ denotes the closed unit ball in ${\cX\times\cY}$ centered at the origin.
It is easy to see from the definition that the Aubin property of ${\rS}_{\rm KKT}$ holds at $(0;0)$ for $(\bar x;\bar y)$ if $(\bar x;\bar y)$ is a strongly regular solution to \eqref{kktop}. 
Moreover, such an Aubin property is equivalent to the metric regularity or the linear openness of $\Phi$ at $(\bar x;\bar y)$ (e.g. \cite[Theorem 9.43]{varbook}).

Since the solution mapping $\rS_{\KKT}$ is implicitly defined, verifying both the strong regularity and the Aubin property from their definitions is generally unachievable.
Consequently, characterizations for the two conditions have evolved into a central topic in optimization theory and variational analysis. 
For conventional nonlinear programming with $\cK$ in \eqref{op} being a convex polyhedral cone, such characterizations have been well established for about three decades. 
Specifically, Robinson \cite{Robinson1980} defined the strong second-order sufficient condition (SSOSC) for the nonlinear programming problem. 
Moreover, at a stationary point $\bar x$ satisfying the linear independence constraint qualification (LICQ), he also showed in \cite[Theorem 4.1]{Robinson1980} that the KKT system is strongly regular at $(\bar x; \bar y)\in \rS_{\rm KKT}$ if the SSOSC holds at $(\bar x;\bar y)$.
At the same time, Kojima \cite{Kojima80} introduced the concept of strong stability to stationary points of the nonlinear programming problem, and showed in \cite[Theorems 6.4 \& 6.5]{Kojima80} that 
for a locally optimal solution satisfying the LICQ, 
it is strongly stable if and only if the SSOSC holds.
It was later revealed by Jongen et al. \cite[Section 3]{jongen87} that the strong regularity and the strong stability are equivalent for stationary points of nonlinear programming where the LICQ holds. 
Furthermore, the strong regularity was characterized by Kummer \cite{kummer90} based on a generalized ``implicit function theorem'' on nonsmooth equations, and one may refer to \cite{Jongen90} and \cite{KlatteTammer} for a related approach, especially \cite[Theorem 4.3]{KlatteTammer} for a survey of equivalent characterizations. 
By combining the results of \cite{Robinson1980}, \cite{Kojima80}, and \cite{KlatteTammer}, one has that at a locally optimal solution of the nonlinear programming problem, the strong regularity is equivalent to the condition that both the SSOSC and the LICQ hold (cf. \cite[Remark 4.11]{Bonans95}). 
Such a result is also available in \cite[Theorem 4.10]{Bonans95} and \cite[Proposition 5.38]{pert}. 
In addition to these equivalent characterizations of the strong regularity, a surprising result of Dontchev and Rockafellar \cite[Theorems 1, 4, \& 5]{don1996} for the nonlinear programming problem is that the Aubin property of ${\rS}_{\rm KKT}$ at $(0;0)$ for $(\bar{x}; \bar{y}) \in \rS_{\rm KKT}(0,0)$, with $\bar{x}$ being a locally optimal solution, implies the strong regularity of the KKT system \eqref{kktop} at $(\bar{x}; \bar{y})$. 
Consequently, a comprehensive class of equivalent characterizations of the Aubin property for the nonlinear programming problem was achieved. 

When the set $\cK$ in \eqref{op} is no longer a polyhedral set, characterizing these two concepts becomes much more involved. 
A more realistic setting is that $\cK$ is $C^2$-cone reducible (cf. \cite[Definition 3.135]{pert}), which is practical enough for encompassing many important classes of optimization problems, including the nonlinear programming, the nonlinear second-order cone programming (NLSOCP), and the nonlinear semidefinite programming (NLSDP) \cite{shapiro03}.
In this setting, the perturbation analysis of problem \eqref{op} has been extensively studied \cite{BCS1,nogap,socpsr,ding,sunmor,varsufD}, 
and characterizing the strong regularity via second-order optimality conditions has been recognized as a prominent topic. 
Although deriving an equivalent characterization of the strong regularity akin to \cite[Proposition 5.38]{pert} for problem \eqref{op} in its general form has not been achieved, achievements have been made for the most representative classes of problems with significant importance in the form of \eqref{op}. 

According to \cite[Theorem 5.24]{pert}, the strong regularity of the KKT system \eqref{kktop} at a solution $(\bar x;\bar y)$ implies that the constraint nondegeneracy condition holds at $\bar x$ (or $\bar x$ is nondegenerate, cf. \eqref{cnd}).
With the help of this result, 
Bonnans and Ram\'irez \cite{socpsr} established a counterpart of \cite[Proposition 5.38]{pert} for the NLSOCP problem.  
For the NLSDP problem, Sun \cite{sunmor} defined the SSOSC by introducing an approximation set, and finally obtained a collection of equivalent characterizations of the strong regularity condition, including the SSOSC accompanied by the constraint nondegeneracy.   
These results properly extend the characterizations of the strong regularity from the conventional nonlinear programming to problem \eqref{op} with $\cK$ being a non-polyhedral set. 
Nevertheless, for both the NLSOCP and the NLSDP problems, obtaining such an extension for characterizing the Aubin property has been an open question for a long time. 

The first step to address this issue was achieved by 
Outrata and Ram\'irez \cite{outrata2011} (and the erratum by Opazo, Outrata, and Ram\'irez \cite{opazo2017}). 
They proved that, for a nondegenerate locally optimal solution $\bar x$ of the NLSOCP problem (i.e., \eqref{op} with $\cK$ being the Cartesian product of second-order cones), the Aubin property of a solution mapping (akin to \eqref{sdelta}) at $0\in\cX$ for $\bar x$ implies the SSOSC, 
but under an assumption regarding the strict complementarity. 
Recently, Chen et al. \cite{chenchensun} finalized this conclusion by removing these assumptions and showed that, with $\bar y\in\cM(\bar x)$, 
the Aubin property of 
$\rS_{\KKT}$ at $(0; 0)$ for $(\bar x;\bar y)$  
is equivalent to the strong regularity of  $(\bar x;\bar y)$ to the KKT system, 
thus constituting for the NLSOCP a counterpart of the seminal result of Dontchev and Rockafellar \cite{don1996}. 
As a result, an extension of the characterizations for the Aubin property is realized from the nonlinear programming problem to the NLSOCP problem. 
Generally, if $\cK$ is a $C^2$-cone reducible set, 
another recent work \cite[Theorem 5.14]{hang} utilized an assumption on relative interiors of subdifferential mappings to prove the equivalence between the Aubin property of $\rS_{\KKT}$ and the strong regularity of the KKT system,   but such an assumption is exactly the strict complementarity condition when applied to the desired NLSDP problem.

In this paper, we study the characterizations of the Aubin property for the KKT solution mapping of the NLSDP problem 
\[
\label{nlsdp}
\min_{x\in\cX} 
f(x)\quad\text{s.t.}
\quad
h(x)=0,
\quad
g_j(x)\in\cS_+^{p_j},
\quad j=1,\ldots, J,
\]
where $h:\cX\to\Re^m$ and $g_j:\cX \to\cS^{p_j}$, $j=1,\ldots,J$ are twice continuously differentiable functions,
$\cS^{p_j}$ is the linear space of ${p_j}\times {p_j}$ real symmetric matrices endowed with the inner product $\langle A, B\rangle:=\trace(A B)$ for $A, B\in\cS^{p_j}$,
where $\trace(\cdot)$ denotes the sum of the diagonal elements, 
and $\cS^{p_j}_+$ is the closed convex cone of all positive semidefinite matrices in $\cS^{p_j}$. When $J=1$ in \eqref{nlsdp}, the strong regularity of the corresponding KKT system is equivalent to the SSOSC together with the constraint nondegeneracy \cite[Theorem 4.1]{sunmor}.
For the case that $J> 1$, it is easy to see from the analysis in \cite{sunmor}
that such an equivalence still holds. 
Moreover, in the case that $J=1$, Fusek \cite{fusek2012} showed that the Aubin property at $(0;0)$ for a solution $(\bar x;\bar y)$ implies the constraint nondegeneracy at $\bar x$. 
Such a result was later extended by Klatte and Kummer \cite{cla} to the general case of problem \eqref{op}. 
 
The main contribution of this paper is that,
at a locally optimal solution $\bar x$ of \eqref{nlsdp} with $\bar y\in\cM(\bar x)$, we prove that the Aubin property of $\rS_{\rm KKT}$ at $(0;0)$ for $(\bar x;\bar y)$ implies the SSOSC. 
We achieve this by designing an auxiliary optimization problem and fully exploiting its properties,
and using the Mordukhovich criterion for characterizing the Aubin property \cite{mordc}.
We should emphasize that the key tools we used in this paper are essentially different from those in \cite{chenchensun} for the NLSOCP.
Specifically, the main progress in \cite{chenchensun} is based on
a lemma of alternative choices on cones, and the fact that the spectral factorization associated with second-order cones admits only two ``eigenvalues'' is indispensable.
This approach can be used here when $\max_j \{p_j\}\le 3$, but does not apply to the general cases. 
Consequently, it is not hard to see that the tools in \cite{chenchensun} are far from sufficient for deriving a counterpart of its main result in the setting of the NLSDP problem \eqref{nlsdp}. 
Based on the equivalence between the Aubin property and the strong regularity established in this paper, 
we finally obtain a comprehensive collection of equivalent characterizations of the Aubin property for the NLSDP problem \eqref{nlsdp}.

The remaining parts of this paper are organized as follows. 
Section \ref{sec:pre} introduces the notation, the definitions, and the preliminary results used throughout this paper. In Section \ref{sec:imp}, we study the implications of the Aubin property of the solution mapping $\rS_{\KKT}$ for NLSDP, especially the SSOSC. 
The characterizations of the Aubin property of $\rS_{\KKT}$ for NLSDP are given in Section \ref{sec:equiv}. 
We conclude this paper with some discussions in Section \ref{sec:conclusion}.

\section{Notation and preliminaries}
\label{sec:pre}
Let $\cE$ and $\cF$ be two finite-dimensional real Hilbert spaces each endowed with an inner 
product $\langle \cdot,\cdot\rangle$ and its induced norm $\|\cdot\|$.
The inner product on $\cE\times\cF$ is defined by 
$\langle (z_1;z_1'), (z_2; z_2')\rangle
:=\langle z_1,z_2\rangle+\langle z_1',z_2'\rangle$ for all
$(z_1;z_2), (z_1'; z_2')\in\cE\times\cF$,
and the norm $\|\cdot\|$ on $\cE \times \cF$ 
is induced by this inner product.
For a vector $z\in\cE$ (or a subspace $\cE_0\subseteq\cE$), $z^\perp$ (or $\cE_0^\perp$) denotes its orthogonal complement in $\cE$.
Given a set of vectors $\{z_1,\dots,z_r\}\subset \cE$, we use $\Span\,\{z_1,\dots,z_r\}$ to denote the linear subspace it spans.
Given a cone $C\subseteq\cE$,
$C^\circ := \{ v \in \cE \mid \langle v, z \rangle\le 0 \  \forall \, z\in C \}$ is the polar cone of $C$. 
For a linear operator $\cA:\cE\to\cF$, we use $\cA^*, \range (\cA)$ and $\ker(\cA)$ to denote its adjoint, range space, and null space, respectively. Note that $\range (\cA)=(\ker (\cA^*))^\perp$.
For a continuously differentiable function $\psi:\cE\to\cF$, we use $\psi'(z)$ to denote the Fr\'echet derivative or the Jacobian of $\psi$ at $z\in\cE$, 
and define $\nabla\psi(z):=(\psi'(z))^*$.

Given a matrix $A\in\Re^{l\times q}$, we use $A_{ik}$ to denote the entry at the $i$-th row and the $k$-th column of $A$, 
and use $A_k$ to denote the $k$-th column of $A$. The transpose of $A$ is denoted by $A^{\top}$.
When $l=q$, we use $A^{-1}$ to denote the inverse of $A$ if it is nonsingular.  
Given a subset $S\subseteq\{1,\ldots, q\}$, we use $|S|$ to denote its cardinality and use $A_{S}$ to denote the sub-matrix of $A$ by eliminating all the columns that are not indexed by $S$ from $A$. 
For the given index sets $I\subseteq\{1,\ldots, l\}$ and $S\subseteq\{1,\ldots, q\}$, 
we use $A_{IS}$ to denote the 
$|I|\times|S|$ sub-matrix of $A$ by removing all the rows and columns not in $I$ and $S$, respectively. Given two matrices $A,B\in\Re^{l\times q}$, $A\circ B$ denotes their Hadamard product. 
The inner product of $A,B$ is defined by $\langle A, B\rangle=\trace(A^\top B)$, and $\|A\|=\sqrt{\langle A, A\rangle}$ is the Frobenius norm.
We use $E$ to denote an all-ones matrix,
whose dimension will be specified from the context.
For a matrix $A\in\cS^p$, we use $A\succ 0$ (or $A\succeq 0$) to say that $A$ is positive definite (or positive semidefinite). 
Moreover, 
$A\prec 0$ (or $A\preceq 0$) means that 
$-A\succ 0$ (or $-A\succeq 0$).

Given a set $C\subseteq\cE$, we use $\lin (C)$ to denote the largest linear subspace contained in $C$ (the lineality space of $C$), $\aff (C)$ to denote the smallest linear subspace that contains $C$ (the affine hull of $C$), 
$\interior\, C$ to denote its topological interior,
and $\closure\, C$ to denote the closure of $C$.
The paratingent cone of $C$ at $\bar z$ is defined by 
$$
\cT^P_{C}(\bar z):=\limsup_{z \stackrel{C}{\to} \bar z, t\searrow 0}\frac{C- z}{t},
$$
where ``$\limsup$'' is the outer limit in Painlev\'e-Kuratowski convergence for subsets, and $z \stackrel{C}{\to} \bar z$ means that $z\to \bar z$ with $z\in C$.  
The regular (Fr\'echet) normal cone of $C$ at $\bar z\in C$ is defined by
$$
\widehat\cN_C(\bar z):=\{v \in\cE\mid \langle v, z-\bar z\rangle\le o(\|z-\bar z\|)  \ \forall\, z\in C\}, 
$$
and the limiting (Mordukhovich) normal cone of $C$ at $\bar z\in C$ is defined by
$$
\cN_C(\bar z):=\limsup_{z \stackrel{C}{\to} \bar z}\widehat\cN_C(z). 
$$
When $C$ is a closed convex set, $\widehat\cN_{C}(\bar z)$ coincides with $\cN_C(\bar z)$, and both of them are simply called the normal cone of $C$ at $\bar z$ (in convex analysis \cite{rocconv}), i.e.,
$$
\cN_C(\bar z)=
\begin{cases}
\{v\in\cE\mid \, \langle v, z-\bar z\rangle\le 0 
\ \forall z\in C\},
&\mbox{if } \bar z\in C,
\\
\emptyset,
&\mbox{otherwise}.
\end{cases}
$$
When $C$ is a closed convex set, the tangent cone $\cT_C(\bar z)$ of $C$ at $\bar z\in C$ can be defined by $\cT_C(\bar z):=(\cN_C(\bar z))^\circ$ (cf. \cite[Example 6.24]{varbook}). 

Given a function $\psi:\cE\to (-\infty,\infty]$ with $\bar z\in\cE$ such that $\psi(\bar z)$ is finite, according to \cite[Theorem 8.9]{varbook}, the (limiting) subdifferential of $\psi$ at $\bar z$ is defined by 
\begin{equation}
\label{subgras}
\partial \psi(\bar z):=\{v\mid (v;-1)\in \cN_{\epi \psi}(\bar z;\psi(\bar z))\},
\end{equation}
where $\epi \psi:=\{(z;t)\mid t\ge \psi(z)\}$ is the epigraph of $\psi$.

Given a set-valued mapping $\Psi:\cE\rightrightarrows\cF$ , we use $\gph\Psi \subseteq\cE\times\cF$ to denote the graph of $\Psi$, i.e.,
$$
\gph\Psi:=\{(z;w)\in\cE\times\cF\mid w\in\Psi(z)\}.
$$
The strict graphical (paratingent) derivative of $\Psi$ at $(\bar z;\bar w)\in\gph\Psi$ is defined by
\begin{equation}\label{def:parde}
    \cD_*\Psi (\bar z,\bar w)(u)=\{v\in\cF\mid (u;v)\in \cT^P_{\gph\Psi}(\bar z;\bar w)\},
    \quad u\in\cE .
\end{equation}
Meanwhile, the 
limiting coderivative of $\Psi$ at $(\bar z;\bar w)\in\gph\Psi$ is defined by
\[
\label{def:cod}
\cD^*\Psi(\bar z,\bar w)(v):=\{u\in\cE\,|\,(u;-v)\in\cN_{\gph \Psi}(\bar z;\bar w)\},
\quad v\in\cF.
\]
The set-valued mapping $\Psi:\cE\rightrightarrows\cF$ is said to have the Aubin property at $\bar z\in\cE$ for $\bar w\in\Psi(\bar z)$ if there exist a constant $\kappa>0$ and open neighborhoods $\cU$ of $\bar z$ and $\cV$ of $\bar w$ such that
$$
\Psi(z)\cap \cV\subseteq\Psi(z')+\kappa\|z-z'\|\mathds{B}_{\cF}\quad
\forall \, z,z'\in \cU,
$$
where $\mathds{B}_{\cF}$ is the closed unit ball in $\cF$ centered at the origin.
Moreover, such an Aubin property has been equivalently characterized by the Mordukhovich criterion \cite{mordc}.
Specifically, under the assumption that $\gph\Psi$ is locally closed around $(\bar z;\bar w)\in\gph\Psi $,
the Aubin property of $\Psi$ holds at $\bar z$ for $\bar w$  if and only if
$\cD^*\Psi(\bar z,\bar w)(0)=\{0\}$.

\subsection{Coderivative related to positive semidefinite cone}
This subsection briefly introduces the explicit formula of the limiting coderivative of the normal cone mapping to $\cS_+^p$, 
which was calculated by Ding et al. \cite{sdpcod}.  
Let $A\in\cS^p$ be given. It admits an eigenvalue decomposition in the form of
\[
\label{eigd}
A=P\Lambda P^\top 
=P\begin{pmatrix}
\Lambda_{\alpha\alpha}\\&0_{\beta\beta}
\\&&\Lambda_{\gamma\gamma}
\end{pmatrix}
P^\top \quad\mbox{with}\quad
P=(P_\alpha,P_\beta,P_\gamma),
\]
where $P\in\Re^{p\times p}$ is an orthogonal matrix 
and $\Lambda$ is a diagonal matrix 
with non-increasing diagonal elements $\lambda_1,\ldots,\lambda_p$ (the eigenvalues of $A$) such that $\Lambda_{\alpha\alpha}\succ 0$ 
and $\Lambda_{\gamma\gamma}\prec 0$.
Here $\alpha,\beta$ and $\gamma$ are three sets of indices with the cardinalities $|\alpha|$, $|\beta|$ and $|\gamma|$. 
One has $P_\alpha\in\Re^{p\times|\alpha|}$, $P_\beta\in\Re^{p\times|\beta|}$ and $P_\gamma\in\Re^{p\times|\gamma|}$.
We use $\Pi_{\cS_+^p}(A)$ to denote the metric projection of $A$ to the cone $\cS_+^p$ (under the Frobenius norm). 
For convenience, we denote $A_+:=\Pi_{\cS_+^p}(A)$ and $A_-:=A-A_+\in -\cS^p_+$. 
Then, it is obvious that $|\alpha|$ and $|\gamma|$ are the ranks of $A_+$ and $A_-$, respectively. 
Recall that $\cS_+^p$ is a closed convex cone, and one has 
$$
\cT_{\cS_+^p}(A_+)
=\{
Z\in\cS^p \,\mid \,
P^\top _{\beta\cup\gamma} ZP_{\beta\cup\gamma}\succeq 0
\}.
$$
Consequently, it holds that 
$$
\lin\big(\cT_{\cS_+^p}(A_+)\big)
=\{
Z\in\cS^p \,\mid \,
P^\top _\beta ZP_\beta=0,\,
P^\top _\beta ZP_\gamma=0,\,
P^\top _\gamma ZP_\gamma=0
\}.
$$
Note that $(A_+, A_-)\in\gph \cN_{\cS_+^p}$, and in the following we will introduce the explicit formulation of the coderivative 
$\cD^*\cN_{\cS_+^p}(A_+, A_-)$. 
For this purpose, denote the set of all partitions of the index set $\beta$ by $\mathscr{P}_{\beta}$ and let $\Re_{\ge}^{|\beta|}$ be the set of all the vectors in $\Re^{|\beta|}$ whose components are arranged in non-increasing order, i.e., 
$$
\Re_{\ge}^{|\beta|}:=\{z\in \Re^{|\beta|}\mid  z_1\ge z_2\ge\dots\ge z_{|\beta|}\}.
$$
For any $z\in\Re_{\ge}^{|\beta|}$, 
one can define $D(z)\in \Re^{|\beta|\times |\beta |}$ as the matrix whose elements $(D(z))_{ik}$, $i,k\in\{1,\ldots,|\beta|\}$, are given by 
$$
(D(z))_{ik}:=
\begin{cases}
\displaystyle \frac{\max\{z_i,0\}-\max\{z_k,0\}}{z_i-z_k}\in[0,1],&\mbox{if } z_i\neq z_k,
\\ 
1,&\mbox{if } z_i=z_k>0,
\\ 
0,& \mbox{if } z_i=z_k\le 0. 
\end{cases}
$$
Define the set 
\[
\label{ubeta}
\Upsilon_{|\beta|}:= 
\left\{\overline Z\in\cS^{|\beta|}
\, \mid \, 
\overline Z
=
\lim_{k\to\infty}D(z^k),\
z^k\to 0,\
z^k\in\Re_{\ge}^{|\beta|}\right\}
\subseteq \cS^{|\beta|}.
\]
Let $\Xi_1\in \Upsilon_{|\beta|}$. 
Then, there exists a partition $\pi(\beta):=(\beta_+,\beta_0,\beta_-)\in \mathscr{P}_{\beta}$ such that
\[
\label{Xi1}
\Xi_1=\begin{pmatrix}
E_{\beta_+\beta_+}&E_{\beta_+\beta_0}&(\Xi_1)_{\beta_+\beta_-}
\\[1mm]
E^\top _{\beta_+\beta_0}&0&0
\\[1mm]
(\Xi_1)^\top _{\beta_+\beta_-}&0&0
\end{pmatrix},
\]
where each element of $(\Xi_1)_{\beta_+\beta_-}$ belongs to the interval $[0,1]$. Moreover, based on $\Xi_1$ one can define the matrix 
\[
\label{Xi2}
\Xi_2:=\begin{pmatrix}
0&0&E_{\beta_+\beta_-}-(\Xi_1)_{\beta_+\beta_-}\\
0&0&E_{\beta_0\beta_-}\\
(E_{\beta_+\beta_-}-(\Xi_1)_{\beta_+\beta_-})^\top &E^\top _{\beta_0\beta_-}& E_{\beta_-\beta_-}
\end{pmatrix}.
\]
Building upon the above definitions, the coderivative of $\cN_{\cS^p_+}$ can be explicitly given in the following result. 
\begin{lemma}{\cite[Theorem 3.1 $\&$ Proposition 3.3]{sdpcod}}
\label{lemma:cod}
Suppose that $A\in \cS^{p}$ has the eigenvalue decomposition in \eqref{eigd}. 
Then,
$U\in\cD^*\cN_{\cS_+^p}(A_+, A_-)(V)$
if and only if
\[
\label{ujg}
U=P\begin{pmatrix}
0&0&\widetilde U_{\alpha\gamma}\\
0&\widetilde U_{\beta\beta}&\widetilde U_{\beta\gamma}\\
\widetilde U_{\gamma\alpha}&\widetilde U_{\gamma\beta}&\widetilde U_{\gamma\gamma}
\end{pmatrix}P^\top  
\quad\mbox{and}\quad
V=P\begin{pmatrix}
\widetilde V_{\alpha\alpha}&\widetilde V_{\alpha\beta}&\widetilde V_{\alpha\gamma}\\
\widetilde V_{\beta\alpha}&\widetilde V_{\beta\beta}&0\\
\widetilde V_{\gamma\alpha}&0&0
\end{pmatrix}P^\top 
\]
with
\[
\label{ujg2}
\widetilde U_{\beta\beta}\in\cD^*\cN_{\cS_+^{|\beta|}}(0,0)(\widetilde V_{\beta\beta})
\quad\mbox{and}\quad
\Sigma_{\alpha\gamma}\circ \widetilde U_{\alpha\gamma}-(E_{\alpha\gamma}-\Sigma_{\alpha\gamma})\circ \widetilde V_{\alpha\gamma}=0,
\]
where $\widetilde U:=P^\top UP\in\cS^p$, $\widetilde V:=P^\top  VP\in\cS^p$, and $\Sigma\in\Re^{p\times p}$ is the matrix defined by 
$$
\Sigma_{ik}:=
\frac{\max\{\lambda_i ,0\}
-\max\{\lambda_{k} ,0\}}{\lambda_i -\lambda_{k} },
\quad i,k\in \{1,\ldots,p\}, 
$$
where $0/0$ is defined to be $1$. 
In addition, 
$\widetilde U_{\beta\beta}\in\cD^*\cN_{\cS_+^{|\beta|}}(0,0)(\widetilde V_{\beta\beta})$ holds if and only if there exist a matrix 
$\Xi_1\in \Upsilon_{|\beta|}$ in \eqref{ubeta}
and an orthogonal matrix $O\in\Re^{|\beta|\times|\beta|}$
such that 
\[
\label{cod0}
\Xi_1\circ O^\top \widetilde U_{\beta\beta}O=\Xi_2\circ O^\top \widetilde V_{\beta\beta}O,
\quad
O_{\beta_0}^\top \widetilde U_{\beta\beta}O_{\beta_0}\preceq 0,
\mbox{ and } 
\quad
O_{\beta_0}^\top \widetilde V_{\beta\beta}O_{\beta_0}\preceq 0,
\]
where $(\beta_+,\beta_0,\beta_-)\in\mathscr{P}(\beta)$ is a partition such that $\Xi_1$
takes the form of \eqref{Xi1} 
and $\Xi_2$ is given by \eqref{Xi2}.
\end{lemma}

We make the following remarks on Lemma \ref{lemma:cod}.

\begin{remark}
According to the definitions of $A_+$ and $A_-$, from \eqref{eigd} one can see that the index set $\beta$ corresponds to the failure of the strict complementarity. 
Specifically, the fact that $A_+\in  \cS^p_+$, $A_-\in \cS^p_-$ and $\langle A_+, A_-\rangle=0$ constitutes a complementarity condition,
and this complementarity is strict if $A$ is not singular.
When $|\beta|\neq 0$, an extensive non-smooth analysis should be imposed to establish the relation between $\widetilde U_{\beta\beta}$ and $\widetilde V_{\beta\beta}$ in \eqref{ujg}, resulting in a more complicated coderivative inclusion \eqref{ujg2} than the relationship for the other matrix pairs. 
\end{remark}

\begin{remark}
\label{rmk:zerocoderivative}
For the case that $A=0$ in Lemma \ref{lemma:cod},  one has $|\beta|=p$. 
Thus, by taking $P$ as the identity matrix, one can get $\widetilde V_{\beta\beta}=V$. 
In addition, one can further take $\beta_+=\beta$ and 
$\beta_-=\beta_0=\emptyset$ to get a partition of $\beta$. 
In this case, from \eqref{Xi1} and \eqref{Xi2} one has
$\Xi_1= E_{\beta\beta}$ and $\Xi_2=0\in\cS^{|\beta|}$. 
Then, one has that \eqref{cod0} holds with  $\widetilde U_{\beta\beta}:=0\in\cS^{|\beta|}$ for any orthogonal matrix $O\in\Re^{|\beta|\times|\beta|}$.
Consequently, 
one has $0\in\cD^*\cN_{\cS_+^p}(0,0)(V)$ for any $V\in\cS^p$. 
\end{remark}

Furthermore, to understand Lemma \ref{lemma:cod} more intuitively, one can consider the following example, which is constructed from \cite[Example 7.1]{sdpcod}. 
\begin{example}
    Consider the matrix 
    \begin{equation*}
    A=\begin{pmatrix}
        0&-2&-1\\ -2&0&-1\\ -1&-1&-1
    \end{pmatrix}
     \text{ with }
    A_+=\begin{pmatrix}
        1&-1&0\\-1&1&0\\0&0&0
    \end{pmatrix}
    \text{ and } A_-=
    \begin{pmatrix}
        -1&-1&-1\\-1&-1&-1\\-1&-1&-1
    \end{pmatrix}.
    \end{equation*}
    We have 
    \begin{equation*}
        A=P\begin{pmatrix}
            2&0&0\\0&0&0\\0&0&-3
        \end{pmatrix}P^{\top}
    \text{  with  }
    P=\begin{pmatrix}
        \frac{1}{\sqrt{2}}& \frac{1}{\sqrt{6}}&\frac{1}{\sqrt{3}}\\
        -\frac{1}{\sqrt{2}}& \frac{1}{\sqrt{6}}&\frac{1}{\sqrt{3}}\\
        0 & -\frac{2}{\sqrt{6}}&\frac{1}{\sqrt{3}}
    \end{pmatrix},
    \end{equation*}
    and that the index sets of positive, zero, and negative eigenvalues are given by $\alpha=\{1\}$, $\beta=\{2\}$, and $\gamma=\{3\}$. Also, we have $\Sigma_{13}=\frac{2}{5}$. Since $|\beta|=1$ and $O\in \{1,-1\}$, three cases should be considered:
    \begin{description}
        \item[(1)] $\beta=\beta_+$, $\Xi_1=1$, and $\Xi_2=0$. One has $\widetilde U_{22}=0$.
        \item[(2)] $\beta=\beta_0$ and $\Xi_1=\Xi_2=0$. From \eqref{cod0} one has $\widetilde U_{22}\le 0$ and $\widetilde V_{22}\le 0$.
        \item[(3)] $\beta=\beta_-$, $\Xi_1=0$, and $\Xi_2=1$. One has $\widetilde V_{22}=0$.
    \end{description}
    Therefore, $U\in \cD^*\cN_{\cS_+^3}(A_+,A_-)(V)$ if and only if
    \begin{equation*}
    V=P\begin{pmatrix}
    \widetilde V_{11}&\widetilde V_{12}&\widetilde V_{13}\\
    \widetilde V_{21}&\widetilde V_{22}&0\\
    \widetilde V_{31}&0&0
    \end{pmatrix}P^\top
\quad\mbox{and}\quad    
    U=P\begin{pmatrix}
    0&0&\widetilde U_{13}\\
    0&\widetilde U_{22}&\widetilde U_{23}\\
    \widetilde U_{31}&\widetilde U_{32}&\widetilde U_{\gamma\gamma}
    \end{pmatrix}P^\top  
    \end{equation*}
    with $2\widetilde U_{13}=3\widetilde V_{13}$ and $(\widetilde U_{22};\widetilde V_{22})\in (\Re\times \{0\})\cup (\{0\}\times \Re) \cup (\Re_-\times\Re_-)$.
\end{example}

\subsection{Second-order sufficient conditions}
In this subsection, we briefly introduce the second-order (sufficient) optimality conditions of the NLSDP problem \eqref{nlsdp}, 
which can be viewed as an extension of the discussions in \cite{sunmor}.

Let $\bar x$ be a stationary point of the NLSDP \eqref{nlsdp}
and $\bar y=(\bar \zeta;\bar\Gamma_1;\ldots;\bar\Gamma_J)\in\cM(\bar x)$ be a multiplier at $\bar x$, where $\bar\zeta\in\Re^m$ and $\bar\Gamma_j\in\cS^{p_j}$ for all $j=1,\dots,J$.
From the KKT system \eqref{kktop} we know that $g_j(
\bar x)\in\cS_+^{p_j}$ and $\bar \Gamma_j\in\cN_{\cS_+^{p_j}}(g_j(\bar x))$.
For convenience, define 
$$g(\bar x):=(g_1(x);\ldots;g_J(\bar x))
\quad \mbox{and}\quad   A_j:=g_j(\bar x)+\bar \Gamma_j,\quad  j=1,\ldots, J.$$
Then one has $(A_j)_+=g_j(\bar x)$ and $(A_j)_-=\bar \Gamma_j$. 
One also has the eigenvalue decomposition $A_j=P_j\Lambda^j P_j^\top $ as in \eqref{eigd} with $P_j=((P_j)_{\alpha_j},(P_j)_{\beta_j},(P_j)_{\gamma_j})$ being the corresponding orthogonal matrix and $\Lambda^j$ being the corresponding diagonal matrix of eigenvalues. 

Recall that at a stationary point $x\in \cX$ of \eqref{op} with $\cM(x)$ being nonempty, the critical cone at $x$ is defined by
$$
\cC(x):=\{d\in \cX \mid G'(x)d\in \cT_{\cK}(G(x)),\, f'(x)d= 0\}.
$$
Then, following the discussions in \cite[Section 3]{sunmor}, the critical cone of \eqref{nlsdp} at $\bar x$ can be explicitly given as
\[
\label{criticalcone}
\cC(\bar x)=\big\{
d \, |\,
h'(\bar x)d=0,\,
g'_j(\bar x)d\in\cC(A_j;\cS_+^{p_j}),
\  j=1,\ldots, J
\big\},
\]
where for each $j=1,\ldots, J$, 
\[
\label{criticalap}
\begin{array}{ll}
\cC(A_j;\cS_+^{p_j}):
=
\cT_{\cS_+^{p_j}}
\big(g_j(\bar x)\big)\cap \bar\Gamma_j^\perp
\\[2mm]
\
=\big\{
Z\in\cS^{p_j}\,|\ 
(P_j)^\top _{\beta_j} Z(P_j)_{\beta_j} \succeq 0,\,
(P_j)^\top _{\beta_j} Z(P_j)_{\gamma_j}=0,\ 
(P_j)^\top _{\gamma_j} Z (P_j)_{\gamma_j}=0
\big\}.
\end{array}
\]
According to \cite[Section 4.6.1]{pert}, for a feasible point $x\in\cX$ of \eqref{op}, it is called nondegenerate \cite{Robinson1984}, or the constraint nondegeneracy \cite{Robinson2003} is called to hold at $x$, if 
\[
\label{cnd}
G'(x)\cX+\lin(\cT_{\cK}(G(x)))=\cY.
\]
Then for the NLSDP \eqref{nlsdp}, the constraint nondegeneracy condition \eqref{cnd} at $\bar x$ can be written as
\[
\label{consnd}
\begin{pmatrix}
h'(\bar x)\\
g'(\bar x)
\end{pmatrix}\mathcal{X}
+
\begin{pmatrix}
\{0\}\\
\prod_{j=1}^{J}\lin\big(\cT_{\cS^{p_j}_+}(g_j(\bar x))\big)
\end{pmatrix}
=
\begin{pmatrix}
\Re^m\\
\prod_{j=1}^J\cS^{p_j}
\end{pmatrix}.
\]
Note that \eqref{consnd} implies that $\cM(\bar x)$ is a singleton, i.e., $\bar y$ is the unique multiplier at $\bar x$.
In this case, by following the proof of \cite[Proposition 3.1]{sunmor} one can get from \eqref{criticalcone} that
\[
\label{critcalcone}
\begin{array}{ll}
\ds
\aff(\cC(\bar x))&=\{d\in\cX\ |\ h'(\bar x)d=0,  g'_j(\bar x)d\in\aff(\cC(A_j;\cS_+^{p_j})),\  j=1,\dots,J\}
\\[2mm]
\ds
&=\left\{d\in\cX\ \Bigg \vert \
\begin{array}{ll}
h'(\bar x)d=0,\ 
(P_j)^\top _{\beta_j} (g'_j(\bar x)d)(P_j)_{\gamma_j}=0,
\\[1mm]
(P_j)^\top _{\gamma_j}  (g'_j(\bar x)d)(P_j)_{\gamma_j}=0,
\ j=1,\dots,J
\end{array}
\right\}.
\end{array}
\]

Next, we discuss the second-order optimality conditions for \eqref{nlsdp}. 
With $\bar x$ being a stationary point of \eqref{nlsdp} and $\bar y=(\bar\zeta; \bar\Gamma_1; \dots; \bar\Gamma_J)\in\cM(\bar x)$, we define 
the self-adjoint linear operator  $\cQ:\cX\to\cX$ by
\[
\label{ssocq}
\langle d, \cQ d\rangle
=
\langle d, \nabla_{xx}^2\cL(\bar x,\bar y) d\rangle-2\sum_{j=1}^J
\langle\bar\Gamma_j, (g'_j(\bar x)d)(g_j(\bar x))^\dag(g'_j(\bar x)d)\rangle,
\quad 
d\in\cX,
\]
where 
$(g_j(\bar x))^\dag$ denotes the Moore–Penrose pseudoinverse of $g_j(\bar x)$ for all $j=1,\ldots, J$.
Based on \eqref{ssocq}, the second-order sufficient condition and the strong second-order sufficient condition (SSOSC) for \eqref{nlsdp} are defined as follows.

\begin{definition}
\label{def:ssosc}
Let $\bar x$ be a stationary point of \eqref{nlsdp} with $\bar y=(\bar\zeta; \bar\Gamma_1; \dots; \bar\Gamma_J)\in\cM(\bar x)$ and $\cC(\bar x)$ be the critical cone at $\bar x$. 
Let $\cQ:\cX\to\cX$ be the self-adjoint linear operator defined by \eqref{ssocq}. 
We say that the second-order sufficient condition holds at $(\bar x;\bar y)$ if
\[
\label{soscnlsdp}
\langle d, \cQ d\rangle>0
\quad \forall d\in \cC(\bar x)\backslash\{0\}.
\]
Moreover, we say that the SSOSC holds at $(\bar x;\bar y)$ if
\[
\label{ssosc}
\langle d, \cQ d\rangle>0 
\quad \forall d\in \aff(\cC(\bar x))\backslash\{0\}. 
\]
\end{definition}
\begin{remark}
In Definition \ref{def:ssosc}, the second-order sufficient condition given by \eqref{soscnlsdp} comes from \cite{nogap}, 
while the SSOSC is a straightforward extension and simplification of \cite[Definition 3.2]{sunmor} from the case that $J=1$ to the general setting.
It is the same as \cite[Definition 3.2]{sunmor} when $J=1$ and the constraint nondegeneracy holds. 
\end{remark}

The following result on the relationship between the Aubin property of $\rS_{\KKT}$ and the second-order sufficient condition \eqref{soscnlsdp} comes from \cite[Corollary 25]{ding}. It constitutes the starting point of our analysis. 
\begin{lemma}
\label{lemma:sosc}
Let $\bar x$ be a locally optimal solution to \eqref{nlsdp} with $\bar y\in\cM(\bar x)$. 
Suppose that the Aubin property of $\rS_{\KKT}$ in \eqref{skkt} holds at $(0;0)$ for $(\bar x;\bar y)$.
Then, $\bar x$ is nondegenerate, in the sense that \eqref{consnd} holds, and the second-order sufficient condition \eqref{soscnlsdp} in Definition \ref{def:ssosc} holds with $\cQ$ being defined in \eqref{ssocq}.
\end{lemma}

\subsection{Technical lemmas} 
In this part, we provide three technical lemmas for our discussions.
These results are not specialized only for the problems considered here.
The first lemma is about the polar cone. 

\begin{lemma}[{\cite[Corollary 11.25(d)]{varbook}}]
\label{polarV}
Let $\cH:\cE\to\cF$ be a linear operator and $K\subseteq\cF$ be a closed nonempty convex cone. 
Then the set $C:=\{z\in \cE\mid \cH z \in K\}$ is also a closed convex cone and one has
\begin{equation*}
C^{\circ}=\closure\,\{ {\cH}^*  v \mid v\in K^{\circ}\},
\end{equation*}
where the closure operation is superfluous if $0$ is in the interior of $\{K-\range\cH\}$.
\end{lemma}

The following lemma gives a variational characterization of self-adjoint positive definite operators. 
One can also see \cite[Proposition 3.1]{jbhu} for a more straightforward proof based on the Moreau decomposition.
\begin{lemma}[{\cite[Theorem 3.6]{magz}}]
\label{lemma:pd}
Let $\cH:\cE\to \cE$ be an invertible self-adjoint linear operator and $C$ be a closed convex cone in $\cE$.
Then
$$
\langle z,\cH z\rangle>0 \quad \forall z\in\cE  \setminus\{0\}
\quad
\Leftrightarrow
\quad
\left\{
\begin{array}{ll}
\langle z, \cH z\rangle>0  \ \forall z\in C\backslash\{0\}\ \mbox{and}
\\[1mm]
\langle z, \cH^{-1} z\rangle >0  \ \forall z\in C^\circ\backslash\{0\}.
\end{array}
\right.
$$
\end{lemma}

The following lemma discusses Robinson's constraint qualification \cite{rcq} of a special constraint system. 
Recall that for the general constraint system $G(x)\in \cK$ in \eqref{op}, 
we say that Robinson's constraint qualification holds at a point $x\in\cX$ such that $G(x)\in \cK$ if 
$0\in{\interior\,}\{
G(x)+G'(x)\cX-\cK\}$. 
According to \cite[Proposition 2.97 \& Corollary 2.98]{pert}, this constraint qualification is equivalent to
$$
G'(x)\cX+\cT_{\cK}(G(x))=\cY.
$$
Moreover, if Robinson's constraint qualification holds at a locally optimal solution $\bar x$ of \eqref{op}, then the set of multipliers $\cM(\bar x)$ is non-empty, convex and compact \cite[Theorem 3.9]{pert}.

\begin{lemma}\label{lemma:rcq}
Let $\cH:\cE\to\cF$ be a linear operator and $K\subseteq\cF$ be a nonempty closed convex cone. 
Then, Robinson's constraint qualification holds at any feasible point to the constraint system 
\[
\label{constraintsystem}
z-\cH^* v=0,  \quad 
\frac{1}{2}(\|v\|^2-1)=0,\quad \mbox{and}\quad 
v\in K. 
\]
\end{lemma}
\begin{proof}
Suppose that  $(\bar z;\bar v)\in \cE\times\cF$ is a given feasible point to the constraint system \eqref{constraintsystem}. Let $(\Delta_z; \delta; \Delta_v)\in \cE\times\Re\times\cF$ be arbitrarily chosen. 
To show that Robinson's constraint qualification holds at $(\bar z;\bar v)$ we need to find $z\in\cE$, $v\in\cF$ and $u \in \cT_K(\bar v)$ such that
\begin{subequations}
\begin{numcases}{}
\label{cond1}
z-\cH^* v=\Delta_z,
\\
\label{cond2}
\langle \bar v, v\rangle =\delta,
\\
\label{cond3}
v+u=\Delta_v.
\end{numcases}
\end{subequations}
Since $K$ is a closed convex cone, 
from \cite[Example 6.24]{varbook} we know that $\cN_K(\bar v)=(\cT_K(\bar v))^{\circ}$.
Then by Moreau's decomposition theorem \cite[Theorem 31.5]{rocconv} we know that one can uniquely decompose $\Delta_v$ by
$\Delta_v =\Delta_v'+\Delta_v''$  such that 
$$
\Delta_v'\in \cT_K(\bar v),
\quad
\Delta_v''\in \cN_K(\bar v),
\quad\mbox{and}\quad 
\langle \Delta_v',\Delta_v''\rangle=0.
$$
One can take $v_0:=\Delta_v''\in \cN_K(\bar v)$ and $u_0:=\Delta_v'\in\cT_{K}(\bar v)$,
so that $v_0+ u_0=\Delta_v$.
Moreover, it is easy to see from $\bar v\in K$ that
$$
\lambda\bar v\in \cT_{K}(\bar v)\quad
\mbox{with}
\quad
\lambda:=\langle \bar v, v_0\rangle-\delta.
$$
Then, by letting $v:=v_0-\lambda \bar v$ and  $u:=u_0+\lambda\bar v$ one has $u\in \cT_{K}(\bar v)$ and that \eqref{cond3} holds.
Moreover, since $\|\bar v\|=1$, we have
$\langle \bar v, v\rangle =\langle \bar v, v_0-\lambda \bar v\rangle=\langle \bar v, v_0\rangle-\lambda=\delta$,
so that \eqref{cond2} is valid.
Then, letting $z:=\Delta_z+\cH^* v$, we have that \eqref{cond1} holds.
This completes the proof.
\end{proof}

\section{Implications of the Aubin property for NLSDP}\label{sec:imp}
In this section, we study the implications of the Aubin property of the solution mapping $\rS_{\KKT}$ in \eqref{skkt} for the NLSDP problem \eqref{nlsdp}. 
Throughout this section, 
we set $G(x):=(h(x);g(x))=(h(x);g_1(x);\ldots;g_J(x))$, 
$\cY:=\Re^m\times \prod_{j=1}^J\cS^{p_j}$ and 
$\cK:=\{0\in\Re^m\}\times\prod_{j=1}^J\cS_+^{p_j}$ in \eqref{op}.
Moreover, we make no distinction between the general constraint optimization problem \eqref{op} and the NLSDP problem \eqref{nlsdp} for convenience. 

\subsection{A reduction method for NLSDP}
This part exploits the Aubin property of the solution mapping $\rS_{\KKT}$ in \eqref{skkt}. 
Note that if $\rS_{\KKT}$ in \eqref{skkt} has the Aubin property at $(0;0)\in \cX\times\cY$ for $(\bar x; \bar y)$, 
the tilt perturbation solution mapping $\rS_{\rm GE}$, defined by
\[
\label{sdelta}
\rS_{\rm GE}(a):=\{
x\in\cX\,|\, a \in\nabla  f(x)+\nabla G(x)\cN_{\cK}(G(x))\},
\]
also has the Aubin property at  $0\in\cX$ for $\bar x$.
Moreover, according to the Mordukhovich criterion \cite{mordc}, the latter Aubin property holds
if and only if 
$$\cD^* {\rS}_{\rm GE}(0, \bar x)(0)=\{0\}.
$$

Next, we use a reduction method to get more accessible results from the original formula of the Mordukhovich criterion. 
Since $\cK$ is $C^2$-cone reducible at every $y\in\cK$,
one can utilize the second-order chain rule developed in \cite[Theorem 7]{outrata2011}, as a generalization of \cite[Theorem 3.4]{chain}, 
to get the following result.

\begin{lemma}
\label{lemma:aubin}
Let $\bar x$ be a stationary point of \eqref{nlsdp} with $\bar y\in\cM(\bar x)$. 
Suppose that $\bar x$ is nondegenerate, i.e., \eqref{consnd} holds. 
Then one has
$$
\cD^* {\rS}_{\rm GE}(0, \bar x)(0)=
\left\{-d 
\mid
0\in \nabla_{x x}^2\cL(\bar x,\bar y)d
+\nabla G(\bar x)\cD^*\cN_\cK(G(\bar x),\bar y)(G'(\bar x)d)
\right\}.
$$
\end{lemma}
\begin{proof}
The proof follows directly from \cite[Theorem 20]{outrata2011}, in which $\cK$ was assumed to be a second-order cone, but 
it still holds when $\cK$ is a general closed convex set. 
\end{proof}

Next, we apply Lemmas \ref{lemma:cod} and \ref{lemma:aubin} to the NLSDP problem \eqref{nlsdp} to obtain the following result.

\begin{proposition}
\label{propaubin}
Let $\bar x$ be a nondegenerate 
stationary point of \eqref{nlsdp}
such that $\bar y=(\bar \zeta; \bar\Gamma_1; \ldots; \bar\Gamma_J)\in \cM(\bar x)$.
For $j\in\{1,\ldots, J\}$, 
define $A_j:=g_j(\bar x)+\bar \Gamma_j$ and write its eigenvalue decomposition $A_j=P_j\Lambda^j P_j^\top $ as in \eqref{eigd}
with $P_j=\big((P_j)_{\alpha_j},(P_j)_{\beta_j},(P_j)_{\gamma_j}\big)$ being the corresponding orthogonal matrix and $\Lambda^j$ being the corresponding diagonal matrix of eigenvalues. 
Then, the mapping $\rS_{\rm GE}$ defined by \eqref{sdelta} has the Aubin property at $0$ for $\bar x$ 
if and only if 
\[
\label{pp1}
\begin{array}{ll}
\displaystyle
\cQ d\notin \range(\nabla h(\bar x))
-\Bigg\{\sum_{j=1}^J \nabla g_j(\bar x) 
(P_j)_{\beta_j\cup\gamma_j}\begin{pmatrix}
\widetilde U_{\beta_j\beta_j}^j &\widetilde U^j_{\beta_j\gamma_j}
\\[2mm]
({\widetilde U}^j_{\beta_j\gamma_j})^\top & 
\widetilde U^j_{\gamma_j\gamma_j}
\end{pmatrix}
(P_j)_{\beta_j\cup\gamma_j}^\top 
\quad 
\\[6mm]
\qquad \bigg\vert\
\begin{array}{ll}
\widetilde U^j_{\beta_j\beta_j}\in\cD^*\cN_{\cS_+^{|\beta_j|}}(0,0)
\big((P_j)_{\beta_j}^\top (g'_j(\bar x) d) (P_j)_{\beta_j}\big)
\\[3mm]
\forall  
{\widetilde  U^j}_{\beta_j\gamma_j}\in\Re^{|\beta_j|\times|\gamma_j|},
{\widetilde  U^j}_{\gamma_j\gamma_j}\in\cS^{|\gamma_j|},
\  j=1,\ldots,J
\end{array}
\Bigg\}\quad   
\forall d\in\aff(\cC(\bar x)) \backslash\{0\},
\end{array}
\]
where $\cQ$ is the linear operator defined by \eqref{ssocq}. 
\end{proposition}

\begin{proof}
From Lemma \ref{lemma:aubin} we know that 
the mapping $\rS_{\rm GE}$ defined by \eqref{sdelta} has the Aubin property at $0$ for $\bar x$ if and only if
\[
\label{aubproof1}
\begin{array}{lll}
\nabla_{xx}^2\cL(\bar x,\bar y)d&\notin&
-
\nabla h(\bar x)\cD^*\cN_{\{0\}}(h(\bar x),\bar \zeta)(h'(\bar x)d)
\\[2mm]
&&
-\sum_{j=1}^J \nabla g_j(\bar x)\cD^*\cN_{\cS_+^{p_j}}(g_j(\bar x),\bar\Gamma_j)(g'_j(\bar x)d)
\quad 
\forall\,0\neq d\in\mathcal{X}.
\end{array}
\]
In the following, we will reformulate the right-hand side of \eqref{aubproof1} into an explicit form by using specific formulas of coderivatives for normal cones.
On the one hand, 
since $h(\bar x)=0$, it is easy to see from the definition of the coderivative in \eqref{def:cod} that 
$$
\cD^*\cN_{\{0\}}(h(\bar x),\bar \zeta)(h'(\bar x)d)
=\begin{cases} 
\Re^m, &\mbox{if }h'(\bar x)d=0, 
\\
\emptyset, &\mbox{otherwise}.
\end{cases}
$$
On the other hand, for all $j=1,\ldots, J$, 
by applying Lemma \ref{lemma:cod} to $A:=A_j=g_j(\bar x)+\bar \Gamma_j$ and $V:=V^j:=g'_j(\bar x)d$ 
we know that 
$U^j\in \cD^*\cN_{\cS_+^{p_j}}(g_j(\bar x),\bar \Gamma_j)(g'_j(\bar x)d)$ 
holds if and only if 
$U:=U^j$ and $V$ can be expressed as in \eqref{ujg}
with $\alpha:=\alpha_j$, $\beta:=\beta_j$ and $\gamma:=\gamma_j$ such that \eqref{ujg2} holds.
Since $\bar x$ is nondegenerate, from the formula of $\aff(\cC(\bar x))$ given in \eqref{critcalcone} one can
see that \eqref{aubproof1} holds 
if and only if 
\[
\label{aubproof2}
\nabla_{xx}^2\cL(\bar x,\bar y) d
\notin
\range(\nabla h(\bar x))
-\sum_{j=1}^J\nabla g_j(\bar x){\mathds U}^j_d 
\quad
\forall d\in\aff(\cC(\bar x))\backslash\{0\}, 
\]
where with $\widetilde V^j_{\alpha_j\gamma_j}:=
(P_j)_{\alpha_j}^\top (g_j'(\bar x)d) (P_j)_{\gamma_j}$ and  
$\widetilde V^j_{\beta_j\beta_j}:=(P_j)^\top _{\beta_j}(g'_j(\bar x)d)(P_j)_{\beta_j}$,  
the set ${\mathds U}^j_d$ is defined by
$$
\begin{array}{rr}
{\mathds U}^j_d:=\Bigg\{
P_j\begin{pmatrix}
0&0& \widetilde U^j_{\alpha_j\gamma_j}\\
0&\widetilde U^j_{\beta_j\beta_j}&\widetilde U^j_{\beta_j\gamma_j}
\\
(\widetilde U^j_{\alpha_j\gamma_j})^\top  
& 
(\widetilde U^j_{\beta_j\gamma_j})^\top  &
\widetilde U^j_{\gamma_j\gamma_j}
\end{pmatrix}P_j^\top 
\qquad\qquad\qquad\qquad\qquad\qquad
\\[8mm]
\Bigg\vert \begin{array}{l}
\widetilde U^j_{\beta_j\beta_j}
\in\cD^*\cN_{\cS_+^{|\beta_j|}}(0,0)
(\widetilde V^j_{\beta_j\beta_j}), 
\\
\Sigma^j_{\alpha_j\gamma_j}\circ  \widetilde  U^j_{\alpha_j\gamma_j}
=(E_{\alpha_j\gamma_j}-\Sigma^j_{\alpha_j\gamma_j})\circ 
\widetilde V^j_{\alpha_j\gamma_j},
\ 
j=1,\ldots, J
\end{array}
\Bigg\},
\end{array}
$$
in which the matrix $\Sigma^j$ is defined by
\[
\label{sigmaj}
\Sigma^j_{ik}:=\frac{\max\{\lambda_i^j ,0\}-\max\{\lambda^j_{k} ,0\}}{\lambda^j_i -\lambda^j_{k} },
\quad  i,k\in\{ 1,\ldots,p_j\},
\ j=1,\ldots, J, 
\] 
where $0/0$ is defined to be $1$ and  each $\lambda^j_i$ is the $i$-th diagonal element of $\Lambda^j$. 
Note that for any $U^j\in\mathds{U}^j_d$, $j=1,\ldots,J$, it holds that
$$
U^j=P_j\begin{pmatrix}
0&0& \widetilde U^j_{\alpha_j\gamma_j}\\
0&\widetilde U^j_{\beta_j\beta_j}
&\widetilde U^j_{\beta_j\gamma_j}
\\
(\widetilde U^j_{\alpha_j\gamma_j})^\top 
&(\widetilde U^j_{\beta_j\gamma_j})^\top 
&{\widetilde U^j}_{\gamma_j\gamma_j}
\end{pmatrix}P_j^\top 
$$
and
$\Sigma^j_{\alpha_j\gamma_j}\circ  \widetilde  U^j_{\alpha_j\gamma_j}
=(E_{\alpha_j\gamma_j}-\Sigma_{\alpha_j\gamma_j})\circ \widetilde V^j_{\alpha_j\gamma_j}$. 
Then, from the definition of $\Sigma^j_{\alpha_j\gamma_j}$ in \eqref{sigmaj}  
one can see that
for any $i\in\alpha_j$ and $k\in\gamma_j$ it holds that
$$
\frac{\lambda^j_i}{\lambda^j_i-\lambda^j_{k}}  
\widetilde U^j_{ik}+\frac{\lambda^j_{k}}{\lambda^j_i-\lambda^j_{k}}
\widetilde V^j_{ik}=0.
$$
Consequently, one has 
${\lambda^j_i} \widetilde U^j_{ik}+ {\lambda^j_{k}}  \widetilde V^j_{ik}=0$.
Therefore, it holds that 
\[
\label{equvlambda}
\widetilde U^j_{\alpha_j\gamma_j}
+(\Lambda^j_{\alpha_j\alpha_j})^{-1}
\widetilde V^j_{\alpha_j\gamma_j}
\Lambda^j_{\gamma_j\gamma_j} =0,
\quad  j=1,\ldots, J.
\]
Next, we provide an equivalent reformulation of the last term\footnote{This term was known as the ``sigma term'' (cf. \cite[p. 177]{socpsr}).} in \eqref{ssocq} for defining the linear operator $\cQ$, for the purpose of studying the left-hand side of \eqref{pp1}.
Note that for any $d\in\cX$ one has $g'_j(\bar x)d\in\cS^{p_j}$ and $(g_j(\bar x))^\dag\in\cS^{p_j}$. 
Therefore, it holds that 
\[
\label{sigma1}
\begin{array}{l}
2\big\langle\bar\Gamma_j, (g'_j(\bar x)d)(g_j(\bar x))^\dag(g'_j(\bar x)d)\big\rangle
\\[2mm]
=
\big\langle   \bar \Gamma_j (g'_j(\bar x)d)(g_j(\bar x))^\dag, g'_j(\bar x)d \big\rangle
+\big\langle  (g_j(\bar x))^\dag(g'_j(\bar x)d)\bar\Gamma_j, g'_j(\bar x)d \big\rangle,
\\[2mm]
=
\big\langle \nabla g_j(\bar x)\big(\bar \Gamma_j (g'_j(\bar x)d)(g_j(\bar x))^\dag
+ (g_j(\bar x))^\dag(g'_j(\bar x)d)\bar\Gamma_j\big), d 
\big\rangle 
\quad\forall j=1,\dots, J.
\end{array}
\]
One can see from the eigenvalue decomposition of $g_j(\bar x)+\bar \Gamma_j=A_j=P_j\Lambda^j P_j^\top $ and  
$\widetilde V^j_{\alpha_j\gamma_j}=(P_j)_{\alpha_j}^\top (g_j'(\bar x)d) (P_j)_{\gamma_j}$ that
\[
\label{sigma2}
\begin{array}{ll}
(g_j(\bar x))^\dag(g_j'(\bar x)d)\bar\Gamma_j
=P_j
\begin{pmatrix}
(\Lambda^j_{\alpha_j\alpha_j})^{-1}
\\
&0
\\
&&0
\end{pmatrix}
P_j^\top 
(g_j'(\bar x)d)
P_j
\begin{pmatrix}
0\\&0
\\&& \Lambda^j_{\gamma_j\gamma_j}
\end{pmatrix}
P_j^\top 
\\[7mm]
\qquad =
P_j\begin{pmatrix}
0&0&(\Lambda^j_{\alpha_j\alpha_j})^{-1}
\widetilde V^j_{\alpha_j\gamma_j}
\Lambda^j_{\gamma_j\gamma_j}
\\
0&0&0
\\
0&0&0
\end{pmatrix}P_j^\top 
=
P_j\begin{pmatrix}
0&0&-\widetilde U^j_{\alpha_j\gamma_j}
\\
0&0&0
\\
0&0&0
\end{pmatrix}P_j^\top ,
\end{array}
\]
where the last equality comes from \eqref{equvlambda}.
Similarly, one also has that 
\[
\label{sigma3}
\bar\Gamma_ j (g_j'(\bar x)d) (g_j(\bar x))^\dag
=
P_j\begin{pmatrix}
0&0&0\\
0&0&0\\
-\widetilde U^j_{\alpha_j\gamma_j}
&0&0
\end{pmatrix}
P_j^\top . 
\] 
Then, by putting \eqref{sigma1}, \eqref{sigma2} and \eqref{sigma3} together,
one can see that 
\[
\label{sigmt}
\begin{array}{l}
2\langle\bar\Gamma_j, (g'_j(\bar x)d)(g_j(\bar x))^\dag(g'_j(\bar x)d)\rangle
=\left\langle
\nabla g_j(\bar x) P_j\begin{pmatrix}
0&0&-\widetilde U^j_{\alpha_j\gamma_j}\\
0&0&0\\
-\widetilde U^j_{\alpha_j\gamma_j}&0&0
\end{pmatrix}
P_j^\top , 
d 
\right\rangle. 
\end{array}
\]
Now, one can see from the definition of $\cQ$ in \eqref{ssocq}  and  \eqref{sigmt} that \eqref{pp1} holds if and only if \eqref{aubproof2} holds.
This completes the proof.
\end{proof}

Proposition \ref{propaubin} may seem complicated and confusing, but in fact, it represents the first step of our reduction method. Taking a look at \eqref{pp1}, only the $\beta_j$ and $\gamma_j$ parts remain, and we have reduced the $\alpha_j$ part in $g_j(\bar x)+\bar\Gamma_j$.
When the assumptions in Proposition \ref{propaubin} hold, 
one can further define the linear operator 
$\cB:\cX\to\Re^m\times \prod_{j=1}^J(\Re^{|\beta_j|\times|\gamma_j|}\times \cS^{|\gamma_j|})$ by
\[
\label{defB}
\cB d:=
\begin{pmatrix}
h'(\bar x)d
\\[1mm]
\big( {2} P^\top _{\beta_1} [g'_1(\bar x)d] P_{\gamma_1};
P^\top _{\gamma_1} [g'_1(\bar x)d] P_{\gamma_1}\big)
\\[1mm]
\vdots
\\[1mm]
\big( {2} P^\top _{\beta_J} [g'_J(\bar x)d] P_{\gamma_J};
P^\top _{\gamma_J} [g'_J(\bar x)d] P_{\gamma_J}\big)
\end{pmatrix},
\quad d\in\cX.
\]
From \eqref{critcalcone} one can see that $\ker(\cB)=\aff(\cC(\bar x))$, which is a finite-dimensional subspace of $\mathcal{X}$. 
Let $r$ be the dimension of $\ker(\cB)$.
Then, a collection of linearly independent vectors $\omega_1,\ldots,\omega_r\in\cX$ can be found such that $\Span\{\omega_1,\ldots,\omega_r\}=\ker(\cB)$. 
Moreover, one can define the linear operator $\cW:\Re^r\to\ker(\cB)$ via 
\[
\label{defw}
\cW\nu=\sum_{i=1}^{r}\nu_i \omega_i,
\quad\nu=(\nu_1;\ldots;\nu_r)\in\Re^r. 
\]
Consequently, $\range(\cW)=\ker(\cB)=\aff(\cC(\bar x))$. 
Based on the definition of $\cW$ in \eqref{defw} we have the following result from Proposition \ref{propaubin}.

\begin{proposition}
\label{aubins}
Suppose that the conditions of Proposition \ref{propaubin} are satisfied.
Let $\cQ$ be the linear operator defined by \eqref{ssocq} and $\cW$ be the linear operator defined by \eqref{defw}. 
Then, the mapping $\rS_{\rm GE}$ defined by \eqref{sdelta} has the Aubin property at $0$ for $\bar x$ if and only if
\[
\label{aubr}
\begin{array}{ll}
\displaystyle
\cW^*\cQ \cW\nu 
\notin\Bigg\{ {-} \sum_{j=1}^J\cW^*\nabla g_j(\bar x) 
(P_j)_{\beta_j}
\widetilde U_{\beta_j\beta_j}^j
(P_j)_{\beta_j}^\top 
\\[5mm]
\displaystyle
\qquad \big\vert\
\widetilde U_{\beta_j\beta_j}\in\cD^*\cN_{\cS_+^{|\beta_j|}}(0,0)((P_j)_{\beta_j}^\top (g'_j(\bar x)\cW\nu)(P_j)_{\beta_j})\Bigg\}
\quad  
\forall\nu\in\Re^r\backslash\{0\}.
\end{array}
\]
\end{proposition}

\begin{proof}
From proposition \ref{propaubin} we know that the the mapping $\rS_{\rm GE}$ defined by \eqref{sdelta} for the NLSDP problem \eqref{nlsdp} has the Aubin property at $0$ for $\bar x$ if and only if \eqref{pp1} holds.
Note that $\cW$ in \eqref{defw} is well-defined and $\range(\cW)= \aff(\cC(\bar x))$.
Therefore, such an Aubin property of $\rS_{\rm GE}$ holds if and only if 
\[
\label{aubpf}
\begin{array}{ll}
\displaystyle
\cQ \cW\nu \notin \range(\nabla h(\bar x))
-\Bigg\{\sum_{j=1}^J\nabla g_j(\bar x) 
(P_j)_{\beta_j\cup\gamma_j}\begin{pmatrix}
\widetilde U_{\beta_j\beta_j}^j &\widetilde U^j_{\beta_j\gamma_j}
\\[2mm]
({\widetilde U}^j_{\beta_j\gamma_j})^\top & 
\widetilde U^j_{\gamma_j\gamma_j}
\end{pmatrix}
(P_j)_{\beta_j\cup\gamma_j}^\top 
\ 
\\[8mm]
\hfill
\ \bigg\vert\
\begin{array}{ll}
\widetilde U^j_{\beta_j\beta_j}\in\cD^*\cN_{\cS_+^{|\beta_j|}}(0,0)
\big((P_j)_{\beta_j}^\top (g'_j(\bar x) \cW\nu ) (P_j)_{\beta_j}\big) 
\\[2mm]
\forall\, 
{\widetilde  U^j}_{\beta_j\gamma_j}\in\Re^{|\beta_j|\times|\gamma_j|},
{\widetilde  U^j}_{\gamma_j\gamma_j}\in\cS^{|\gamma_j|},
\  j=1,\ldots,J
\end{array}
\Bigg\} 
\quad  
\forall\nu\in\Re^r\backslash\{0\}.
\end{array}
\]
Consequently, it is sufficient to prove that $\eqref{aubpf}$ and \eqref{aubr} are equivalent.

First, suppose that \eqref{aubr} does not hold, 
i.e., 
there exists a nonzero vector $\bar\nu \in\Re^r$ such that
\[
\label{aubf}
\cW^*\cQ \cW \bar\nu
=
-\sum_{j=1}^J\cW^*\nabla g_j(\bar x) 
(P_j)_{\beta_j}
\widetilde U_{\beta_j\beta_j}^j
(P_j)_{\beta_j}^\top 
\]
with
$$
\widetilde U_{\beta_j\beta_j}\in\cD^*\cN_{\cS_+^{|\beta_j|}}(0,0)((P_j)_{\beta_j}^\top (g'_j(\bar x)\cW\nu)(P_j)_{\beta_j}),
\quad j=1,\ldots, J.  
$$
Denote $\bar \mu:=\cQ \cW\bar\nu$. 
From \eqref{aubf} one has 
\[
\label{u1in}
\bar\mu+
\sum_{j=1}^J\nabla g_j(\bar x) 
(P_j)_{\beta_j}
\widetilde U_{\beta_j\beta_j}^j
(P_j)_{\beta_j}^\top \in\ker(\cW^*).
\]
Recall that the linear operator $\cB$ in \eqref{defB} is well-defined and 
$$\ker(\cW^*)= {(\range(\cW))^\perp=}(\ker(\cB))^\perp=\range(\cB^*).
$$
Thus, we calculate $\range(\cB^*)$ to further reformulate \eqref{u1in}. 
Note that for any $d\in\cX$, $\xi\in\Re^m$,
$U^j_{\beta_j\gamma_j}\in \Re^{|\beta_j|\times |\gamma_j|}$,
$U^j_{\gamma_j\gamma_j}\in\cS^{|\gamma_j|}$, 
$j=1,\ldots, J$, 
it holds from \eqref{defB} that
$$
\begin{array}{l}
\ds
\left\langle\cB^*\left(\xi; \big(U^1_{\beta_j\gamma_j};U^1_{\gamma_j\gamma_j}\big);\cdots; 
\big(U^J_{\beta_J\gamma_J};U^J_{\gamma_J\gamma_J}\big)
\right), d\right\rangle
\\[2mm]
=\left\langle \left(\xi; \big(U^1_{\beta_j\gamma_j};U^1_{\gamma_j\gamma_j}\big);\cdots; 
\big(U^J_{\beta_J\gamma_J};U^J_{\gamma_J\gamma_J}\big)
\right),\cB d\right \rangle
\\[4mm] 
\ds
=\langle \xi, h'(\bar x)d\rangle
\\ 
\qquad 
\displaystyle 
+\sum_{j=1}^J
\left\langle 
U^j_{\beta_j\gamma_j}, 
2(P_j)^\top _{\beta_j} [g_j'(\bar x)d] (P_j)_{\gamma_j}
\right\rangle
+\sum_{j=1}^J 
\left\langle 
U^j_{\gamma_j\gamma_j}, (P_j)^\top _{\gamma_j} 
[g_j'(\bar x)d] (P_j)_{\gamma_j}
\right\rangle
\\[6mm] 
\ds
=\langle \nabla h(\bar x)\xi, d\rangle
\\
\qquad 
\displaystyle 
+2\sum_{j=1}^J
\left\langle 
(P_j)_{\beta_j}  U^j_{\beta_j\gamma_j} (P_j)_{\gamma_j}^\top, 
g_j'(\bar x)d
\right\rangle
+\sum_{j=1}^J \left\langle 
(P_j)_{\gamma_j} U^j_{\gamma_j\gamma_j}(P_j)_{\gamma_j}^\top, 
g_j'(\bar x)d  \right\rangle. 
\end{array}
$$
We know that $g_j'(\bar x)d$ is symmetric for all $j=1,\ldots J$. 
Therefore, one has 
\[
\label{bstar}
\begin{array}{l}
\ds
\left\langle\cB^*\left(\xi; \big(U^1_{\beta_j\gamma_j};U^1_{\gamma_j\gamma_j}\big);\cdots; 
\big(U^J_{\beta_J\gamma_J};U^J_{\gamma_J\gamma_J}\big)
\right), d\right\rangle
\\[4mm]
\ds
=\langle \nabla h(\bar x)\xi, d\rangle
+\sum_{j=1}^J
\left\langle 
(P_j)_{\gamma_j} (U^j_{\beta_j\gamma_j})^\top (P_j)^\top_{\beta_j}, 
g_j'(\bar x)d 
\right\rangle
\\
\displaystyle
\qquad
+\sum_{j=1}^J
\left\langle 
(P_j)_{\beta_j} U^j_{\beta_j\gamma_j} (P_j)^\top_{\gamma_j}, 
g_j'(\bar x)d 
\right\rangle
+\sum_{j=1}^J
\left\langle 
(P_j)_{\gamma_j} U^j_{\gamma_j\gamma_j} (P_j)^\top_{\gamma_j}, 
g_j'(\bar x)d 
\right \rangle
\\ 
\displaystyle 
=\Big\langle \nabla h(\bar x)\xi
+\sum_{j=1}^J \nabla g_j(\bar x) (P_j)_{\gamma_j} (U^j_{\beta_j\gamma_j})^\top (P_j)^\top_{\beta_j}
\\ 
\qquad 
\displaystyle
+\sum_{j=1}^J \nabla g_j(\bar x) (P_j)_{\beta_j} U^j_{\beta_j\gamma_j} (P_j)^\top_{\gamma_j}
+\sum_{j=1}^J\nabla g_j(\bar x)
(P_j)_{\gamma_j} U^j_{\gamma_j\gamma_j} (P_j)^\top _{\gamma_j}, d \Big\rangle.
\end{array}
\]
Then by combining \eqref{u1in} and the formula of $\range(\cB^*)\equiv \ker(\cW^*)$ given by \eqref{bstar} one can get 
$$
\begin{array}{ll}
\displaystyle
\bar\mu \in 
\range\big(\nabla h(\bar x)\big)
-\Bigg\{
\sum_{j=1}^J\nabla g_j(\bar x) 
(P_j)_{\beta_j\cup\gamma_j}\begin{pmatrix}
\widetilde U_{\beta_j\beta_j}^j &\widetilde U^j_{\beta_j\gamma_j}
\\[2mm]
({\widetilde U}^j_{\beta_j\gamma_j})^\top & 
\widetilde U^j_{\gamma_j\gamma_j}
\end{pmatrix}
(P_j)_{\beta_j\cup\gamma_j}^\top 
\qquad 
\\[8mm]
\displaystyle
\hfill
\ \bigg\vert\
\begin{array}{ll}
\widetilde U^j_{\beta_j\beta_j}\in\cD^*\cN_{\cS_+^{|\beta_j|}}(0,0)
\big((P_j)_{\beta_j}^\top (g'_j(\bar x) \cW\bar \nu ) (P_j)_{\beta_j}\big)
\\[2mm]
\forall\, 
{\widetilde  U^j}_{\beta_j\gamma_j}\in\Re^{|\beta_j|\times|\gamma_j|},
{\widetilde  U^j}_{\gamma_j\gamma_j}\in\cS^{|\gamma_j|},
\ j=1,\ldots,J
\end{array}
\Bigg\},
\end{array}
$$
which contradicts \eqref{aubpf} since  $\bar \mu=\cQ \cW\bar\nu$. 
Therefore, $\eqref{aubpf}$ implies \eqref{aubr}.

Next, suppose that $\eqref{aubpf}$ does not hold,  
i.e., there exist two vectors $\tilde \nu \in\Re^r$ and  $\xi\in\Re^m$, 
the matrices
$$
{\widetilde  U^j}_{\beta_j\gamma_j}\in\Re^{|\beta_j|\times|\gamma_j|},
\quad 
{\widetilde  U^j}_{\gamma_j\gamma_j}\in\cS^{|\gamma_j|},
\quad 
j=1,\ldots,J, 
$$
and the matrices
$$\widetilde U^j_{\beta_j\beta_j}\in\cD^*\cN_{\cS_+^{|\beta_j|}}(0,0)
\big((P_j)_{\beta_j}^\top (g'_j(\bar x) \cW\tilde  \nu ) (P_j)_{\beta_j}\big),\quad j=1,\ldots,J,$$ 
such that 
$$
\cQ \cW\tilde \nu  =
\nabla h(\bar x)\xi
-
\sum_{j=1}^J
\nabla g_j(\bar x) 
(P_j)_{\beta_j\cup\gamma_j}\begin{pmatrix}
\widetilde U_{\beta_j\beta_j}^j &\widetilde U^j_{\beta_j\gamma_j}
\\[2mm]
({\widetilde U}^j_{\beta_j\gamma_j})^\top & 
\widetilde U^j_{\gamma_j\gamma_j}
\end{pmatrix}
(P_j)_{\beta_j\cup\gamma_j}^\top .
$$
Note that for any vector $\nu \in\Re^r$, it holds that  $\cW \nu\in\ker(\cB)=\aff(\cC(\bar x))$. 
Then by using \eqref{critcalcone} one has for all  $\nu \in\Re^r$ it holds that 
$$
\begin{array}{ll}
\langle \nu, \cW^*\cQ \cW \tilde \nu \rangle 
\\[2mm]
=
\langle \nu,
\cW^*\nabla h(\bar x)\xi\rangle
\\[2mm]
\qquad 
-\sum_{j=1}^J
\left\langle \nu,
\cW^*\nabla g_j(\bar x) 
(P_j)_{\beta_j\cup\gamma_j}
\begin{pmatrix}
\widetilde U_{\beta_j\beta_j}^j &\widetilde U^j_{\beta_j\gamma_j}
\\[1mm]
({\widetilde U}^j_{\beta_j\gamma_j})^\top & 
\widetilde U^j_{\gamma_j\gamma_j}
\end{pmatrix}
(P_j)_{\beta_j\cup\gamma_j}^\top 
\right\rangle
\\[6mm]
=
\langle  h'(\bar x)\cW \nu,\xi\rangle
\\[1mm]
\qquad 
-\sum_{j=1}^J
\left\langle
(P_j)_{\beta_j\cup\gamma_j}^\top (g_j'(\bar x)\cW \nu)(P_j)_{\beta_j\cup\gamma_j},
\begin{pmatrix}
\widetilde U_{\beta_j\beta_j}^j &\widetilde U^j_{\beta_j\gamma_j}
\\[2mm]
({\widetilde U}^j_{\beta_j\gamma_j})^\top & 
\widetilde U^j_{\gamma_j\gamma_j}
\end{pmatrix}
\right\rangle
\\[6mm]
=
-\sum_{j=1}^J
\left\langle
(P_j)_{\beta_j}^\top (g_j'(\bar x)\cW \nu)(P_j)_{\beta_j}
,
\widetilde U_{\beta_j\beta_j}^j \right\rangle
\\[3mm]
=
-\sum_{j=1}^J
\left\langle
\nu,
\cW^*\nabla g_j(\bar x)
(P_j)_{\beta_j} \widetilde U_{\beta_j\beta_j}^j (P_j)_{\beta_j}^\top  \right\rangle,
\end{array}
$$
which means that $\cW^*\cQ \cW \tilde \nu 
=
-\sum_{j=1}^J
\cW^*\nabla g_j(\bar x)
(P_j)_{\beta_j} \widetilde U_{\beta_j\beta_j}^j (P_j)_{\beta_j}^\top $.
Therefore, \eqref{aubr} also implies $\eqref{aubpf}$.
Consequently, \eqref{aubr} is equivalent to $\eqref{aubpf}$, which completes the proof.
\end{proof}

Propositions \ref{propaubin} and \ref{aubins} constitute the reduction method by reformulating the Mordukhovich criterion to characterize the Aubin property of
$\rS_{\rm GE}$ defined by \eqref{sdelta} to that given by \eqref{aubr}, in which only the indices $\beta_j$ are involved while the indices $\alpha_j$ and $\gamma_j$ are eliminated. 
Such a reduction provides a more accessible form of the Aubin property, which is convenient for the discussions on obtaining the SSOSC from the second-order sufficient condition.  
Before that, we provide a useful result based on Proposition \ref{aubins}.

\begin{corollary}
\label{0inin}
Under the conditions of Proposition \ref{aubins}, the linear operator $\cW^*\cQ \cW$ is a nonsingular matrix in $\cS^r$. 
\end{corollary}

\begin{proof} 
Note that $\cW^*\cQ \cW\in\cS^r$ holds by definition. 
Suppose on the contrary that there exists a nonzero vector $\bar\nu\in\Re^r$ 
such that $\cW^*\cQ \cW \bar\nu =0$. 
According to the discussions in Remark \ref{rmk:zerocoderivative} one can set 
$$
\widetilde U^j_{\beta_j\beta_j}:=0
\in\cD^*\cN_{\cS_+^{|\beta_j|}}(0,0)((P_j)_{\beta_j}^\top (g_j'(\bar x)\cW \bar\nu)(P_j)_{\beta_j}),
\quad
j=1,\dots,J.
$$ 
As a result, one has 
$\cW^*\cQ \cW \bar\nu =0=-\sum_{j=1}^J\cW^* \nabla g_j(\bar x)(P_j)_{\beta_j}\widetilde U^j_{\beta_j\beta_j}(P_j)^\top _{\beta_j}$, 
which contradicts \eqref{aubr}.
Consequently, $\cW^*\cQ \cW$ is not singular. 
\end{proof}

\subsection{Aubin property implies SSOSC}
Note that Proposition \ref{aubins} in the previous subsection provides a characterization of the Aubin property of the mapping $\rS_{\rm GE}$ defined by \eqref{sdelta} using only the indices $\beta_j$. 
In this part, our objective is to utilize \eqref{aubr} and the second-order sufficient condition in \eqref{soscnlsdp} (as a consequence of the Aubin property of $\rS_{\KKT}$ for locally optimal solutions by Lemma \ref{lemma:sosc}) to derive the SSOSC in \eqref{ssosc}. 
For this purpose, we need to first reformulate the two second-order optimality conditions to pave the way for using Proposition \ref{aubins}.

Under the conditions of Proposition \ref{aubins}, define a linear operator $\cA:\Re^r\to \cS^{|\beta_1|}\times\dots\times\cS^{|\beta_J|}$ by 
\[
\label{defanui}
\cA\nu:=(\cA_1\nu;\dots;\cA_J\nu),
\]
where the linear operators $\cA_j:\Re^r\to\cS^{|\beta_j|}$, $j=1,\ldots, J$, are given by 
$$
\cA_j \nu:=(P_j)_{\beta_j}^\top  (g'_j(\bar x)\cW \nu )(P_j)_{\beta_j},  \quad \nu\in\Re^r.
$$
Then, one can define the closed convex cone $\Omega\subseteq \Re^r$ by
$$
\Omega:=\{\nu \in\Re^r \, |\, 
\cA_j \nu \succeq 0,\ j=1,\ldots, J\}.
$$
Note that $\bar x$ is nondegenerate, and the linear operators $\cB$ in \eqref{defB} and $\cW$ in \eqref{defw} are well defined with $\range(\cW)=\ker(\cB)=\aff(\cC(\bar x))$. 
Thus, for any $d\in\aff(C(\bar x))$ there exists a vector $\nu\in\Re^r$ such that $d=\cW\nu$. 
Consequently, the SSOSC \eqref{ssosc} can be equivalently written as
\[
\label{ssoscsdp2}
\langle \nu,\cW^*\cQ \cW \nu\rangle>0 \quad \forall\nu \in\Re^r\backslash\{0\}.
\]
Moreover, recall from \eqref{criticalcone}, \eqref{criticalap} and \eqref{critcalcone} that
$$
\begin{array}{ll}
\cC(\bar x)
& =
\left\{d\in\cX\ \Bigg \vert \
\begin{array}{ll}
h'(\bar x)d=0,\,(P_j)^\top _{\beta_j} (g'_j(\bar x)d)(P_j)_{\beta_j}\succeq 0,
\\[1mm]
(P_j)^\top _{\beta_j} (g'_j(\bar x)d)(P_j)_{\gamma_j}=0,
\\[1mm]
(P_j)^\top _{\gamma_j}  (g'_j(\bar x)d)(P_j)_{\gamma_j}=0,
\,\, j=1,\dots,J
\end{array}
\right\}
\\[6mm]
& =\left\{d\in\cX \ \big\vert\  d\in \aff(\cC(\bar x)),\, 
(P_j)^\top _{\beta_j} (g'_j(\bar x)d)(P_j)_{\beta_j}\succeq 0,
j=1,\dots, J
\right\}.
\end{array}
$$
Therefore, the second-order sufficient condition \eqref{soscnlsdp} can be equivalently recast as
\[
\label{soscsdp2}
\langle \nu,\cW^*\cQ \cW \nu \rangle>0\quad \forall\nu \in \Omega\backslash\{0\}.
\]
Note that by Corollary \ref{0inin}, the matrix $\cW^*\cQ \cW$ is nonsingular. 
Then, by using Lemma \ref{lemma:pd} and \eqref{soscsdp2} we know that \eqref{ssoscsdp2} holds if and only if 
\begin{equation}
\label{redssosc}
\langle \eta,(\cW^*\cQ \cW )^{-1} \eta \rangle >0  \quad \forall\, \eta \in \Omega^{\circ}\backslash\{0\}.
\end{equation} 
In the following, we show that the conditions of Lemma \ref{lemma:sosc} can actually imply \eqref{ssoscsdp2}, hence the SSOSC given by \eqref{ssosc}, which constitutes a remarkable improvement to the second-order sufficient condition in this Lemma.
To achieve this, we need the explicit formula of $\Omega^{\circ}$, and the following result is essential. 

\begin{lemma}
\label{lempolar}
Under the conditions of Proposition \ref{aubins}, the linear operator $\cA$ defined by \eqref{defanui} is surjective. 
\end{lemma}
\begin{proof}
Since the constraint nondegeneracy \eqref{consnd} holds, for any given
$Y_j\in\cS^{p_j},\,j=1,\dots,J$, one can always find a vector $d\in\cX$
and the matrices
$$
\begin{array}{ll}
Z_j&\in\lin\big(\cT_{\cS_+^{p_j}}(g_j(\bar x))\big)
\\[1mm]
&=\{
Z_j\in\cS^{p_j}  \mid  (P_j)^\top _{\beta_j\cup\gamma_j} Z_j(P_j)_{\beta_j\cup \gamma_j}=0\},
\quad  j=1,\ldots, J 
\end{array}
$$
such that 
$$
\begin{cases}
h'(\bar x) d=0, 
\\
g'_j(\bar x) d+Z_j = P_j\begin{pmatrix} 
(Y_j)_{\alpha_j\alpha_j} & (Y_j)_{\alpha_j\beta_j} & (Y_j)_{\alpha_j\gamma_j}
\\
(Y_j)_{\beta_j\alpha_j} & (Y_j)_{\beta_j\beta_j} & 0
\\
(Y_j)_{\gamma_j\alpha_j} & 0 & 0
\end{pmatrix}
P_j^\top, \quad   j=1,\dots, J.
\end{cases}
$$ 
Consequently, it holds that 
$$
\begin{array}{ll}
P_j^\top (g'_j(\bar x) d)P_j
+
\begin{pmatrix}
(P_j)_{\alpha_j}^\top  Z_j (P_j)_{\alpha_j} & (P_j)_{\alpha_j}^\top  Z_j (P_j)_{\beta_j} &
(P_j)_{\alpha_j}^\top  Z_j (P_j)_{\gamma_j}
\\
(P_j)_{\beta_j}^\top  Z_j (P_j)_{\alpha_j} & 0 & 0
\\
(P_j)_{\gamma_j}^\top  Z_j (P_j)_{\alpha_j} & 0 & 0
\end{pmatrix}
\\[6mm]
= 
\begin{pmatrix} 
(Y_j)_{\alpha_j\alpha_j} & (Y_j)_{\alpha_j\beta_j} & (Y_j)_{\alpha_j\gamma_j}
\\
(Y_j)_{\beta_j\alpha_j} & (Y_j)_{\beta_j\beta_j} & 0
\\
(Y_j)_{\gamma_j\alpha_j} & 0 & 0
\end{pmatrix},
\quad   j=1,\dots,J.
\end{array}
$$ 
Since $(Y_j)_{\beta_j\beta_j}$ can be any matrix in $\cS^{|\beta_j|}$,  one can see that the linear operator $\cA$ defined by \eqref{defanui} is subjective.
This completes the proof. 
\end{proof}

Since $\cA$ is subjective by Lemma \ref{lempolar}, one can explicitly calculate from Lemma \ref{polarV} that 
\begin{equation}
\label{omegapolar}
\begin{aligned}
\Omega^\circ
&=
\{ \cA^*(\Theta_1;\dots;\Theta_J) \mid 
\Theta_j\in -\cS^{|\beta_j|}_+,\ \,j=1,\dots,J\}\\
&=
\left\{
\sum_{j=1}^J \cW^*\nabla g_j(\bar x)(P_j)_{\beta_j}\Theta_j (P_j)^\top _{\beta_j}\,|\,\Theta_j\preceq 0,\  j=1,\ldots, J 
\right\}.
\end{aligned}
\end{equation}

Now, we are ready to present the main result of this section.
 
\begin{theorem}
\label{sdpssosc}
Let $\bar x$ be a locally optimal solution to the NLSDP problem \eqref{nlsdp} and $\bar y=(\bar \zeta;\bar \Gamma_1;\dots;\bar\Gamma_J)\in \cM(\bar x)$ be a multiplier at $\bar x$.
Suppose that the solution mapping ${\rS}_{\rm KKT}$ in \eqref{skkt} has the Aubin property at $(0;0)$ for $(\bar x;\bar y)$.
Then, the SSOSC \eqref{ssosc} in Definition \ref{def:ssosc} holds at $(\bar x;\bar y)$ with $\cQ$ being defined in \eqref{ssocq}.
\end{theorem}

\begin{proof}
Following the notation of Proposition \ref{aubins}, for $j\in\{1,\ldots, J\}$, define $A_j:=g_j(\bar x)+\bar \Gamma_j$ and write its eigenvalue decomposition $A_j=P_j\Lambda^j P_j^\top $ as in \eqref{eigd} with $P_j=\big((P_j)_{\alpha_j},(P_j)_{\beta_j},(P_j)_{\gamma_j}\big)$ being the corresponding orthogonal matrix and $\Lambda^j$ being the corresponding diagonal matrix of eigenvalues. 
Then, from the above analysis, we only need to prove that \eqref{redssosc} holds.

We start by considering the following auxiliary optimization problem
\begin{equation}
\label{redmodel}
\begin{array}{cl}
\min\limits_{\eta\in\Re^r,\Theta_j\in\cS^{|\beta_j|}} 
\quad 
&\frac{1}{2}\langle \eta,(\cW^*\cQ \cW )^{-1}\eta\rangle 
\\ 
\mbox{s.t. }&\begin{cases}
\ 
\eta-\sum_{j=1}^J\cW^* \nabla g_j(\bar{x}) (P_j)_{\beta_j}\Theta_j (P_j)_{\beta_j}^\top =0,
\\[1mm]
\frac{1}{2}(\sum_{j=1}^J\|\Theta_j\|^2-1)=0, \\[1mm]
\Theta_j\preceq 0,\;
j=1, \dots, J.
\end{cases} 
\end{array}
\end{equation}
Since the feasible set of \eqref{redmodel} is compact, the minimum of the objective function can be attained at a certain solution 
$(\bar\eta;\bar\Theta_1;\dots;\bar\Theta_J)$. 
Moreover, by applying Lemma \ref{lemma:rcq} with $\cH=\cA,\, K=-(\cS^{|\beta_1|}_+\times\dots\times\cS^{|\beta_J|}_+),\,\cE=\Re^r$ and $\cF=\cS^{|\beta_1|}\times\dots\times\cS^{|\beta_J|}$, one has that Robinson's constraint qualification holds for the constraint in \eqref{redmodel} at $(\bar\eta; \bar\Theta_1; \dots; \bar\Theta_J)$. Consequently, we know from  \cite[Theorem 3.9]{pert} that 
there exists a Lagrange multiplier 
$(\bar\rho;\bar\tau; \bar\Delta_1; \dots; \bar\Delta_J)\in \Re^r\times \Re\times \cS^{|\beta_1|}\times\dots\times\cS^{|\beta_J|}$ 
at $(\bar\eta; \bar\Theta_1; \dots; \bar\Theta_J)$ 
such that the following KKT system of \eqref{redmodel} holds:
\begin{equation}
\label{redKKT}
\begin{cases}
(\cW^*\cQ \cW )^{-1}\bar \eta+\bar\rho=0,
\\[2mm]
-(P_j)_{\beta_j}^\top (g_j'(\bar{x})\cW \bar\rho)(P_j)_{\beta_j}+\bar\tau \bar\Theta_j
+\bar\Delta_j=0,
\\[2mm]
\bar \eta-\sum_{j=1}^J\cW^* \nabla g_j(\bar{x}) (P_j)_{\beta_j}\bar \Theta_j (P_j)_{\beta_j}^\top =0, 
\\[2mm]
\frac{1}{2}(\sum_{j=1}^J\|\bar \Theta_j\|^2-1)=0, 
\\[2mm]
\bar\Theta_j\preceq 0,\,
\bar\Delta_j\succeq 0, 
\langle \bar\Delta_j,\bar\Theta_j\rangle=0,\;  j=1,\dots,J.
\end{cases}
\end{equation}
It is obvious that
\[
\label{qequality}
\cW^*\cQ \cW \bar\rho 
= -\bar \eta 
= -\sum_{j=1}^J\cW^* \nabla g_j(\bar{x}) (P_j)_{\beta_j}\bar\Theta_j (P_j)_{\beta_j}^\top . 
\]
Moreover, from \eqref{redKKT} one can also see that  $\sum_{j=1}^J\|\bar \Theta_j\|^2=1$, so that 
\begin{equation}
\label{deftau}
\begin{aligned}
\bar\tau
&=\langle (\bar\Delta_1;\dots;\bar\Delta_J)+\bar\tau(\bar\Theta_1;\dots;\bar\Theta_J),(\bar\Theta_1;\dots;\bar\Theta_J)\rangle
\\[1mm]
&=\sum_{j=1}^J\left\langle (P_j)_{\beta_j}^\top (g_j'(\bar{x})\cW \bar\rho)(P_j)_{\beta_j}, \bar\Theta_j\right\rangle\\
&=\Big\langle \bar\rho,\sum_{j=1}^J\cW^* \nabla g_j(\bar{x}) (P_j)_{\beta_j}\bar\Theta_j (P_j)_{\beta_j}^\top 
\Big\rangle 
=\langle \bar\rho,\bar\eta\rangle =-\langle \bar\eta, (\cW^*\cQ \cW )^{-1}\bar\eta\rangle.
\end{aligned}
\end{equation}
According to the last line of \eqref{redKKT} one has the eigenvalue decompositions
\begin{equation*}
\bar\Theta_j+\bar\Delta_j
=\bar O_j\begin{pmatrix}
(\bar\Lambda_j)_{\bar \alpha_j\bar \alpha_j}&0&0\\
0&0_{\bar\beta_j\bar\beta_j}&0\\
0&0&(\bar\Lambda_j)_{\bar\gamma_j\bar\gamma_j}
\end{pmatrix}
\bar O_j^\top ,\quad j=1,\dots,J,
\end{equation*}
where each $\bar O_j$ is an orthogonal matrix 
such that $(\bar\Lambda_j)_{\bar \alpha_j\bar \alpha_j}\succ 0$ and $(\bar\Lambda_j)_{\bar\gamma_j\bar\gamma_j}\prec 0$. 
In this case, it holds that
$$
\bar\Theta_j
=\bar O_j\begin{pmatrix}
0_{\bar\alpha_j\bar\alpha_j}&0&0\\
0&0_{\bar\beta_j\bar\beta_j}&0\\
0&0&(\bar\Lambda_j)_{\bar\gamma_j\bar\gamma_j}
\end{pmatrix}
\bar O_j^\top ,\quad j=1,\dots,J
$$
and 
\[
\label{deltatautheta}
\bar\Delta_j+\bar\tau \bar\Theta_j
=\bar O_j\begin{pmatrix}
(\bar\Lambda_j)_{\bar \alpha_j\bar \alpha_j}&0&0\\
0&0_{\bar\beta_j\bar\beta_j}&0\\
0&0&\bar\tau (\bar\Lambda_j)_{\bar\gamma_j\bar\gamma_j}
\end{pmatrix}
\bar O_j^\top ,\quad j=1,\dots,J.
\]
For an arbitrarily given index $j\in\{1,\ldots, J\}$, one can take $\beta_+=\bar\alpha_j$, $\beta_0=\bar\beta_j\cup\bar\gamma_j$ and $\beta_-=\emptyset$ in \eqref{Xi1} and \eqref{Xi2} to set 
$$
\Xi_1=\begin{pmatrix}
E_{\bar\alpha_j\bar\alpha_j}&E_{\bar\alpha_j\bar\beta_j}& E_{\bar\alpha_j\bar\gamma_j}
\\
E_{\bar\beta_j\bar\alpha_j}&0&0
\\
E_{\bar\gamma_j\bar\alpha_j} &0&0
\end{pmatrix}
\quad\mbox{and}\quad 
\Xi_2
=0\in\cS^{|\beta_j|}.
$$
Consequently, it holds that 
\[
\label{precondition}
\begin{cases}
\Xi_1\circ \bar O^\top  \bar\Theta_j \bar O
=0
=\Xi_2\circ 
\bar O^\top  (\bar\Delta_j+\bar\tau \bar\Theta_j) \bar O,
\\
\bar O_{\beta_0}^\top    \bar\Theta_j    \bar O_{\beta_0}\preceq 0.
\end{cases}
\]
Now, suppose on the contrary that \eqref{redssosc} does not hold.
With the help of the explicit formula of $\Omega^\circ$ in \eqref{omegapolar}, 
it is easy to see that the optimal value of \eqref{redmodel} is not positive, so that 
$\bar\tau \ge 0$ by \eqref{deftau}. 
Thus, by \eqref{deltatautheta} one has
$$
\bar O_{\beta_0}^\top  (\bar\Delta_j+\bar\tau\bar\Theta_j) \bar O_{\beta_0}\preceq 0, 
$$
which, together with \eqref{precondition} and Lemma \ref{lemma:cod}, implies that 
$$
\bar\Theta_j\in \cD^*\cN_{\cS_+^{|\beta_j|}}(0,0)(\bar\Delta_j+\bar\tau\bar\Theta_j)
=
\cD^*\cN_{\cS_+^{|\beta_j|}}(0,0)((P_j)_{\beta_j}^\top (g_j'(\bar{x})\cW \bar\rho)(P_j)_{\beta_j}), 
$$
where the equality holds from the second line of \eqref{redKKT}. 
 Note that such an inclusion holds simultaneously for all $j=1,\dots, J$. 
Thus, this inclusion, together with \eqref{qequality}, makes a contradiction to \eqref{aubr} in Proposition \ref{aubins} (with $\nu= \bar\rho$ and $\widetilde U^j_{\beta_j\beta_j}=\bar\Theta_j$ for all $j=1,\ldots, J$). 
Consequently, we know that \eqref{redssosc} is valid, which completes the proof. 
\end{proof}

\section{Characterizations of the Aubin property for NLSDP}
\label{sec:equiv}

This section establishes the equivalent characterizations of the Aubin property of $\rS_{\KKT}$ in \eqref{skkt} at $(0;0)$ for $(\bar x;\bar y)$ with $\bar x$ being a locally optimal solution to the NLSDP problem \eqref{nlsdp} and $\bar y\in\cM(\bar x)$.

We first review some related concepts in variational analysis. 
As mentioned in Section \ref{sec:intro}, the Aubin property is related to the strong metric regularity \cite[Definition 2.5]{smr2022}.
Specifically, for a set-valued mapping $\Psi:\cE\rightrightarrows\cF$, one has $\Psi$ is strongly metrically regular at $(\bar z;\bar w)\in \gph\Psi$ if $\Psi^{-1}$ has the Aubin property at $\bar w$ for $\bar z$, and there exist neighborhoods $\cU$ of $\bar z$ and $\cV$ of $\bar w$, such that $\Psi^{-1}(w)\cap \cU$ is a singleton for all $w \in \cV$. The following result provides a criterion for characterizing the strong metric regularity. 

\begin{lemma}{{\cite[Theorem 2.7]{smr2022}}}
\label{lemma:smr}
$\Psi$ is strongly metrically regular at $(\bar z;\bar w)\in \gph(\Psi)$ if and only if for all $w\in \cF$ and $z\in \cE$, one has
$$
0\in \cD^*\Psi(\bar z,\bar w)(w)\,\Rightarrow\,w=0\quad \text{and}\quad 0\in \cD_*\Psi(\bar z,\bar w)(z)\,\Rightarrow\,z=0,
$$
where $\cD_*$ refers to the strict graphical derivative defined in \eqref{def:parde}.
\end{lemma}

Recall that the KKT system \eqref{kktop} can be equivalently written as the nonsmooth equation 
\begin{equation}
\label{DefF}
    F(x,y):=\begin{pmatrix}
        \nabla_x\cL(x,y)\\
        -G(x)+\Pi_{\cK}(G(x)+y)
    \end{pmatrix}
    =0.
\end{equation}
Since $F$ is locally Lipschitz continuous around $(\bar x;\bar y)$, it is almost everywhere differentiable in a neighborhood $\cV$ of $(\bar x;\bar y)$ by Rademacher's theorem \cite[Theorem 9.60]{varbook}. We use $D_F\subseteq \cV$ to denote the set of points at which $F$ is differentiable.
The Bouligand subdifferential of $F$ at $(\bar x;\bar y)$ is defined by 
$$
\partial_B F(\bar x,\bar y):=\{ v\in\cX\times\cY \mid \exists \,(x_k; y_k)\stackrel{D_F}\longrightarrow(\bar x;\bar y) \text{ with } F'(x_k,y_k)\to v\}. 
$$
Moreover, the Clarke generalized Jacobian of $F$ at $(\bar x;\bar y)$ is defined by 
$$
\overline{\partial}F(\bar x,\bar y):=\conv(\partial_B F(\bar x,\bar y)),
$$
i.e., the convex hull of $\partial_B F(\bar x,\bar y)$.

The perturbed KKT system \eqref{pop} corresponds to a two-parametric optimization problem
\begin{equation}
\label{paraop}
\min_{ x\in\cX}
\  \phi(x,b)-\langle a,x\rangle,
\end{equation}
where 
\[
\label{defphixb}
\phi(x,b):=f(x)+\delta_{\cK}(G(x)+b)
\]
with $\delta_{\cK}(\cdot)$ being the indicator function of $\cK$ in convex analysis \cite{rocconv}.
When $a=0\in\cX$ and $b=0\in\cY$, one has \eqref{paraop} is exactly \eqref{op}. 
Given  $\iota>0$ and $(\bar x;\bar b)\in\cX\times\cY$ such that $\phi(\bar x,\bar b)$ is finite, one can define
\begin{equation}
\label{defml}
\begin{cases}
\displaystyle    M_\iota (a,b):=\argmin_{x\in\cX}\{\phi(x,b)-\langle a,x\rangle\mid \|x-\bar x\|\le \iota \},\\
\displaystyle    m_\iota(a,b):=\inf_{x\in\cX}\{\phi(x,b)-\langle a,x\rangle\mid \|x-\bar x\|\le \iota \}.
\end{cases}
\end{equation}
We say the point $\bar x$ is a Lipschitzian fully stable local minimizer \cite[Definition 3.2]{fullstab} of \eqref{paraop} at $(\bar a; \bar b)$ if there exist numbers $\kappa>0, \iota>0$ and a neighborhood $\cV$ of $(\bar a;\bar b)$ such that the mapping $M_\iota (a,b)$ is single-valued on $\cV$ with $M_\iota (\bar a,\bar b)=\bar x$ satisfying
\begin{equation*}
\|M_\iota (a_1,b_1)-M_\iota (a_2,b_2)\|\le \kappa\|(a_1;b_1)-(a_2;b_2)\| \quad\forall\, (a_1;b_1),(a_2;b_2)\in \cV, 
\end{equation*}
and that the function $m_\iota (a,b)$ is also Lipschitz continuous on $\cV$.
Unlike $\rS_{\rm KKT}$, the mapping $M_\iota (a,b)$ focuses mainly on locally optimal solutions.

Recently, Rockafellar \cite{roc19,varsufR} introduced the strong variational sufficient condition for local optimality 
and provided several characterizations of this abstract property. 
With $\phi$ defined in \eqref{defphixb}, one can reformulate \eqref{op} to 
$$
\min_{x\in\cX,b\in\cY} \, \phi(x,b)\ \text{ s.t. }\ b=0. 
$$
Define the function $\phi_\ell(x,b):=\phi (x,b)+\frac{\ell}{2}\|b\|^2$.
The following definition comes from \cite[Section 2]{varsufR}.

\begin{definition}\label{def:varsuf}
The (strong) variational sufficient condition for local optimality in \eqref{op} holds with respect to a solution 
$(\bar x;\bar y)$ to the KKT system \eqref{kktop} if there exists $\ell >0$ such that $\phi_\ell$ is (strongly) variationally convex with respect to the pair $((\bar x;0),(0;\bar y)) \in \gph (\partial \phi_\ell)$, 
i.e., 
there exist open convex neighborhoods
 $\cU$ of $(\bar x;0)$ and $\cV$ of $(0;\bar y)$, and a closed proper (strongly) convex function
$\varrho\leq\phi_\ell$ on $\cU$
such that 
\begin{equation*}
    (\cU\times \cV)\cap \gph(\partial\varrho)=(\cU\times \cV) \cap \gph(\partial\phi_\ell),
\end{equation*}
and $\varrho (x, b) = \phi_\ell (x, b)$ holds for $((x;b);(a;y))$ belonging to this common set. 
Here, $\partial\varrho$ is the subdifferential in convex analysis and $\partial\phi_\ell$ is the limiting subdifferential defined in \eqref{subgras}.
\end{definition}

\begin{remark}\label{rem:varsuf}
For a special case ($J=1$ and $m=0$) of the NLSDP problem \eqref{nlsdp}, the strong variational sufficient condition (with respect to a solution $(\bar x;\bar y)$ to the KKT system \eqref{kktop}) and the SSOSC \eqref{ssosc} (for the same $(\bar x;\bar y)$) were proved to be equivalent in \cite[Theorem 3.3]{varsufD}. 
This equivalence can be extended to the NLSDP problem \eqref{nlsdp} in its general form with ease by directly following their proof.
\end{remark}

More recently, the primal-dual full stability was studied in \cite{benko24} as an extension of the above (primal) full stability. Specifically, given $\bar y\in \cM(\bar x)$, in addition to $M_{\iota}$ in \eqref{defml}, one defines
\begin{equation*}
    \overline M_{\iota}(a,b):=\{(x,y)\mid x\in M_{\iota}(a,b),\, (a,y)\in \partial\phi(x,b),\,\|y-\bar y\|\le \iota\}.
\end{equation*}
We say that the primal-dual pair $(\bar x;\bar y)$ is fully stable \cite[Definition 1.4]{benko24} in problem \eqref{nlsdp} if there exist a number $\iota>0$ and a neighborhood $\cV$ of $(0;0)$ such that the mapping $\overline M_{\iota}$ is single-valued and Lipschitz continuous in $\cV$, and the function $m_{\iota}$ is likewise Lipschitz continuous on $\cV$.

Based on the above definitions, the following result holds regarding the equivalent characterizations of the Aubin property for the solution mapping ${\rS}_{\rm KKT}$ in \eqref{skkt} of the NLSDP problem \eqref{nlsdp}.

\begin{theorem}
\label{theo:equi}
Let $\bar x$ be a locally optimal solution to the NLSDP problem \eqref{nlsdp} and $\bar y=(\bar \zeta;\bar \Gamma_1;\dots;\bar\Gamma_J)\in \cM(\bar x)$ be a multiplier at $\bar x$.
Then, the following  statements are equivalent:
\begin{enumerate}[align=left, label=\rm ({\roman*})]
\item The solution mapping ${\rS}_{\rm KKT}$ in \eqref{skkt} has the Aubin property at $(0;0)$ for $(\bar x;\bar y)$. 
\item The strong second-order sufficient condition \eqref{ssosc} holds at $(\bar x; \bar y)$ and $\bar x$ is nondegenerate, i.e., \eqref{consnd} holds. 
\item  Any element of the Clarke generalized Jacobian $\overline\partial F(\bar x,\bar y)$ is nonsingular, where the function $F$ is defined in \eqref{DefF}.
\item The KKT point $(\bar x;\bar y)$ is a strongly regular solution to the generalized equation \eqref{def:phi} (or the KKT system \eqref{kktpop}).
\item 
The mapping $\Phi$ in \eqref{def:phi} is strongly metrically regular at $(\bar x;\bar y)$ for $(0;0)$. 

\item   For any $w\in \cX\times \cY$ with  $0\in \cD^*\Phi(\bar x,\bar y)(w)$, one has $w=0$.

\item   The strong variational sufficient condition in Definition \ref{def:varsuf} holds with respect to $(\bar x;\bar y)$, and $\bar x$ is nondegenerate, i.e., \eqref{consnd} holds. 
\item   $\bar x$ is a Lipschitzian fully stable local minimizer of \eqref{paraop}, and $\bar x$ is nondegenerate, i.e., \eqref{consnd} holds.

\item The primal-dual pair $(\bar x;\bar y)$ is fully stable in \eqref{nlsdp}. 

\end{enumerate}
\end{theorem}

\begin{proof}
One has $\rm (i)\Rightarrow\rm (ii)$ from Theorem \ref{sdpssosc} and \cite[Theorem 1]{cla}. 
By simply repeating the proof of \cite[Proposition 3.2]{sunmor} one can get $\rm (ii)\Rightarrow\rm (iii)\Rightarrow\rm (iv)$.  
It follows by \cite[Corollary 2.2]{Robinson1980} and the definition that $\rm (iv)\Rightarrow\rm (v)$ holds.
From Lemma \ref{lemma:smr} and the Mordukhovich criterion for the Aubin property \cite[Theorem 9.40]{varbook} one has $\rm (v)\Rightarrow \rm (vi)\Rightarrow (i)$.
According to Remark \ref{rem:varsuf} we know that ${\rm(ii)} \Leftrightarrow{\rm (vii)}$. 
One also has from \cite[Theorem 5.6]{fullstab} that ${\rm(iv)}\Leftrightarrow{\rm (viii)}$. 
Moreover, from \cite[Theorems 2.3 \& 4.2]{benko24} we have $\rm (v)\Leftrightarrow\rm (ix)$. 
This completes the proof. 
\end{proof}

In the following, we introduce an example to help with understanding Theorem \ref{theo:equi}, focusing on the equivalence of (i), (ii), and (iii).

\begin{example}
Consider the following optimization problem
\begin{equation}\label{eg2}
    \begin{aligned}
        \min_{x\in \cS^3} \quad &f(x):=\frac{1}{2}\|x\|^2-\frac{1}{2}x_{11}^2 \\
        \text{s.t. }\quad &h(x):=x_{11}-1=0, \\
        &g(x):=x\in \cS_+^3.
    \end{aligned}
\end{equation}
The unique optimal solution $\bar x$ of \eqref{eg2} and its unique Lagrange multiplier $\bar y$ are given by 
\begin{equation*}
    \bar x=\begin{pmatrix}
        1&0&0\\0&0&0\\0&0&0
    \end{pmatrix} 
    \quad
    \mbox{and}\quad
    \bar y=(\bar\zeta,\bar\Gamma)\quad\mbox{with}\quad 
    \bar \zeta=0, \quad 
    \bar \Gamma=\begin{pmatrix}
        0&0&0\\0&0&0\\0&0&0
    \end{pmatrix}.
\end{equation*}
It is easy to verify by definition that the constraint nondegeneracy holds at $\bar x$ for \eqref{eg2}. 
By direct calculations the linear operator $\cQ$ and $\cB$ defined in \eqref{ssocq} and \eqref{defB} are given by
\begin{equation*}
    \cQ d=\begin{pmatrix}
        0&d_{12}&d_{13}\\d_{21}&d_{22}&d_{23}\\d_{31}&d_{32}&d_{33}
    \end{pmatrix}\quad\mbox{and}\quad 
    \cB d=d_{11}\quad\forall d\in \cS_3. 
\end{equation*}
Note that one can find a basis for $\ker\cB$ given by 
\begin{equation*}
\begin{array}{ll}
    \{\omega_1,\omega_2,\omega_3,\omega_4,\omega_5\}
    \\[2mm]=
    \left\{\begin{pmatrix}
        0&1&0\\1&0&0\\0&0&0
    \end{pmatrix},
    \begin{pmatrix}
        0&0&1\\ 0&0&0 \\1&0&0
    \end{pmatrix},
    \begin{pmatrix}
        0&0&0\\0&1&0\\0&0&0
    \end{pmatrix},
    \begin{pmatrix}
        0&0&0\\0&0&1\\0&1&0
    \end{pmatrix},
    \begin{pmatrix}
        0&0&0\\0&0&0\\0&0&1
    \end{pmatrix}\right\}. 
    \end{array}
\end{equation*}
Taking this basis for the definition of $\cW$ in \eqref{defw}, one can get 
\[
\label{wqwexample}
\cW^*\cQ\cW \nu=(2\nu_1;2 \nu_2; \nu_3; 2\nu_4; \nu_5)\quad \forall \nu=(\nu_1; \nu_2; \nu_3; \nu_4; \nu_5)\in \Re^5.
\]
Thus the SOSC \eqref{soscsdp2} holds, and one can use the Aubin property of $\rS_{\rm GE}$ to obtain the SSOSC. 

According to Proposition \ref{aubins}, the Aubin property of $\rS_{\rm GE}$  at $0$ for $\bar x$ can be equivalently expressed as that for all $0\neq \nu\in \Re^5$, 
\begin{equation}\label{eg2aubin}
    \begin{pmatrix}
        2\nu_1\\2\nu_2\\ \nu_3\\2\nu_4\\ \nu_5
    \end{pmatrix}\notin \left\{\begin{pmatrix}
        0\\0\\-u_{22}\\-u_{23}-u_{32}\\-u_{33}
    \end{pmatrix} \Big \vert\,
    \begin{pmatrix}
        u_{22}&u_{23}\\u_{32}&u_{33}
    \end{pmatrix}\in \cD^*\cN_{\cS_+^2}(0,0)
    \begin{pmatrix}
        \nu_3&\nu_4\\ \nu_4&\nu_5
    \end{pmatrix}\right\}. 
\end{equation}
If  \eqref{eg2aubin} does not hold, there exists a certain nonzero $\nu\in\Re^5$ such that \begin{equation*}
    \nu_1=\nu_2=0\quad\mbox{and}\quad  -\begin{pmatrix}
        \nu_3&\nu_4\\ \nu_4&\nu_5
    \end{pmatrix}\in \cD^*_{\cS_+^2}(0,0)
    \begin{pmatrix}
        \nu_3&\nu_4\\ \nu_4&\nu_5
    \end{pmatrix}.
\end{equation*}
However, this is impossible due to \cite[Theorem 2.1]{pol98}. 
Thus, the Aubin property of $\rS_{\rm GE}$ at $0$ for $\bar x$ holds.
Consequently, the analysis in Theorem \ref{sdpssosc} tells us that the SSOSC \eqref{ssoscsdp2} holds. 
Of course, one can also directly observe the SSOSC from \eqref{wqwexample}. 
As a result, the solution mapping $\rS_{\KKT}$ also has the Aubin property at the origin for $(\bar x;\bar y)$.

Besides, for the function $F$ defined by \eqref{DefF}, from \cite[Lemma 11]{pangsunsun} and \cite[Lemma 2.1]{sunmor} one has for any $W\in \overline{\partial} F(\bar x,\bar \zeta,\bar \Gamma)$, there exists a matrix $T\in \cS^3$ such that 
\begin{equation*}
    W(\Delta x,\Delta\zeta,\Delta\Gamma)=\begin{pmatrix}
        \begin{pmatrix}
            \Delta\zeta&\Delta x_{12}&\Delta x_{13}\\
            \Delta x_{21}&\Delta x_{22}&\Delta x_{23}\\
            \Delta x_{31}&\Delta x_{32}&\Delta x_{33}
        \end{pmatrix}
        +\Delta\Gamma\\
        \Delta x_{11}\\
        -\Delta x+ T \circ(\Delta x+\Delta\Gamma)
    \end{pmatrix},
\end{equation*}
where ``$\circ$" denotes the Hadamard product and 
\begin{equation*}
T\in\Bigg\{\begin{pmatrix}
1&1&1\\1& t_1& t_2\\1& t_2& t_3
\end{pmatrix}\Big \vert\, t_1, t_2, t_3\in [0,1]\Bigg\}.
\end{equation*}
Thus, it is easy to see that each element in $\overline{\partial}F(\bar x,\bar\zeta,\bar\Gamma)$ is not singular.

\end{example}

\begin{remark}
Note that the nine equivalent conditions listed in Theorem \ref{theo:equi} are not exhaustive. 
For example,
one additional equivalent condition, according to \cite[Remark 3.1]{sunmor}, is that $F$ is a locally Lipschitz homeomorphism \cite[Definition 2.1.9]{facbook} near $(\bar x;\bar y)$.
Furthermore, by \cite[Theorem 4.1]{sunmor}, another condition is that $\bar x$ is nondegenerate and strongly stable \cite[Definition 5.33]{pert}. For more details and equivalent conditions, one may refer to \cite{sunmor} and the references therein.
\end{remark}

\begin{remark}

According to the equivalence between the two conditions \rm (v) and \rm (vi) in Theorem \ref{theo:equi}, one can see from Lemma \ref{lemma:smr} and \cite[Section 2]{benko24} that the conditions in Theorem \ref{theo:equi} are also equivalent to 
    \begin{equation*}
        0\in \cD_*\Phi(\bar x,\bar y)(z) \quad \Rightarrow\quad  z=0 
        \quad \forall z\in \cX\times\cY,
    \end{equation*}
\end{remark}
where $\cD_*\Phi$ denotes the strict graphical derivative defined by \eqref{def:parde}.

\begin{remark}
It should be emphasized that, for \eqref{op} with $\cK$ being an arbitrary $C^2$-cone reducible set, 
it is still unknown if the strong regularity of the KKT system \eqref{def:phi} is equivalent to 
the constraint nondegeneracy \eqref{cnd} combined with a certain second-order optimality condition similar to \eqref{ssosc}. 
The currently known cases include the nonlinear programming \cite{don1996}, the NLSOCP \cite{socpsr}, the NLSDP problem \cite{sunmor} and a composite matrix programming regarding matrix eigenvectors \cite{cuiding}.  
\end{remark}

\section{Conclusions}
\label{sec:conclusion}
In this paper, we prove that at a locally optimal solution to the nonlinear semidefinite programming problem \eqref{nlsdp},
the Aubin property of $\rS_{\rm KKT}$ \eqref{skkt} is equivalent to the strong second-order sufficient condition plus the constraint nondegeneracy. 
This enables us to derive a series of equivalent characterizations of the Aubin property, which includes the strong regularity of the Karush-Kuhn-Tucker system \eqref{kktop}. 
As a byproduct, for nonlinear semidefinite programming, this paper answers the open question posed in \cite[Section 5]{ding} if the Aubin property can be characterized by an exact form of a certain second-order optimality condition together with the constraint nondegeneracy. 
It should be noted that our analysis for nonlinear semidefinite programming (and also for nonlinear second-order cone programming in \cite{chenchensun}) relies on the explicit formulas of the coderivative for the underlying normal cone mapping. 
Currently, it is not clear to us how to extend these results to generic non-polyhedral $C^2$-cone reducible constrained optimization problems. 
We leave this as our future research topic.

\end{document}